\documentclass[10pt]{amsart}
\usepackage{geometry}
\usepackage{amsmath}
\usepackage{amsthm}
\usepackage{hyperref}
\usepackage{ragged2e}
\usepackage{enumitem}
\usepackage{tikz-cd}
\usepackage{mathtools}
\usepackage{thmtools}

\usepackage{amssymb}
\newcommand{\bb}{\mathbb}

\newcommand{\Gal}{\operatorname{Gal}}
\newcommand{\Lap}{\triangle}

\newcommand{\an}{\operatorname{an}}

\newcommand{\supp}{\operatorname{supp}}

\newcommand{\Diag}{\operatorname{Diag}}

\newcommand{\ovl}{\overline}

\newcommand{\berkA}{{\mathbb{A}^1_{\operatorname{Berk},v}}}

\newcommand{\cal}{\mathcal}

\newcommand{\Prep}{\operatorname{Prep}}
\newcommand{\rad}{\operatorname{rad}}

\newcommand{\denom}{\operatorname{denom}}
\newcommand{\dd}{\operatorname{d}}

\newcommand{\dist}{\operatorname{dist}}
\newcommand{\eps}{\varepsilon}
\newcommand{\del}{\partial}
\newcommand{\mc}[1]{\mathcal{#1}}
\newcommand{\Qbar}{{\overline{\bb Q}}}
\newcommand{\PxP}{\mc P_c(X)\times\mc P(X)}
\newcommand{\Rat}{\operatorname{Rat}}
\usepackage{cleveref}

\newtheorem{theorem}{Theorem}[section]
\newtheorem{lemma}[theorem]{Lemma}
\newtheorem{proposition}[theorem]{Proposition}
\newtheorem{corollary}[theorem]{Corollary}
\newtheorem{conjecture}[theorem]{Conjecture}

\theoremstyle{definition}
\newtheorem{definition}[theorem]{Definition}

\theoremstyle{remark}
\newtheorem*{remark}{Remark}

\crefname{section}{Section}{Sections}
\crefname{subsection}{Section}{Sections}
\crefname{theorem}{Theorem}{Theorems}
\crefname{lemma}{Lemma}{Lemmas}
\crefname{proposition}{Proposition}{Propositions}
\crefname{corollary}{Corollary}{Corollaries}
\crefname{definition}{Definition}{Definitions}

\usepackage[style=ext-alphabetic,articlein=false]{biblatex}
\DeclareFieldFormat[article, inbook]{title}{#1}
\makeindex
\title{Shared Dynamically-Small Points for Polynomials on Average}
\author{Yan Sheng Ang}
\email{\href{mailto:angyansheng@yahoo.com.sg}{angyansheng@yahoo.com.sg}}
\author{Jit Wu Yap}
\email{\href{mailto:jyap@math.harvard.edu}{jyap@math.harvard.edu}}

\begin{document}

\begin{abstract}
Given two rational maps $f,g: \bb{P}^1 \to \bb{P}^1$ of degree $d$ over $\bb{C}$, DeMarco--Krieger--Ye \cite{DKY22} has conjectured that there should be a uniform bound $B = B(d) > 0$ such that either they have at most $B$ common preperiodic points or they have the same set of preperiodic points. We study their conjecture from a statistical perspective and prove that the average number of shared preperiodic points is zero for monic polynomials of degree $d \geq 6$ with rational coefficients. We also investigate the quantity $\liminf_{x \in \ovl{\bb{Q}}} \left(\widehat{h}_f(x) + \widehat{h}_g(x) \right)$ for a generic pair of polynomials and prove both lower and upper bounds for it. 
\end{abstract}

\maketitle
\section{Introduction}
Let $f: \bb{P}^1 \to \bb{P}^1$ be a rational function of degree $d \geq 2$ defined over a field $K$. The set of preperiodic points of $f$, denoted by $\Prep(f)$, is the set of points $x \in \bb{P}^1(\ovl{K})$ such that $f^m(x) = f^n(x)$ for two distinct integers $m > n \geq 0$, where $\ovl{K}$ denotes an algebraic closure of $K$. Our starting point is the following finiteness theorem for the common preperiodic points of two rational maps of degree $d \geq 2$.
\begin{theorem}[{\cites[Cor.\ 1.3]{baker-demarco}[Thm.\ 1.3]{yuan-zhang}}]\label{baker-demarco}
For any pair of rational functions $f,g: \bb{P}^1 \to \bb{P}^1$ defined over $\bb{C}$ of degree $d\geq2$, we have
\[|\Prep(f)\cap\Prep(g)|<\infty\quad\text{or}\quad\Prep(f)=\Prep(g).\]
\end{theorem}

We remark that the case of $f,g$ being defined over $\ovl{\bb{Q}}$ had been proven previously by Mimar \cite{mimar}. We now consider two possible directions in which we can study extensions of this result. The first is to ask if the number of common preperiodic points can be uniformly bounded in families of pairs $(f,g)$:
\begin{conjecture}[{\cite[Conj.\ 1.4]{DKY22}}]
Let $d\geq2$. There exists a constant $B=B(d)$ so that for every pair of rational functions $f,g: \bb{P}^1 \to \bb{P}^1$ defined over $\bb{C}$ of degree $d$, we have
\[|\Prep(f)\cap\Prep(g)|\leq B\quad\text{or}\quad\Prep(f)=\Prep(g).\]
\end{conjecture}
In this direction, DeMarco--Krieger--Ye \cite[Thm.\ 1.1]{DKY22} obtain an effective uniform bound on the family of quadratic polynomials of the form $f(z)=z^2+c_1$, $g(z)=z^2+c_2$ with $c_1,c_2 \in \bb{C}$. Mavraki--Schmidt \cite[Thm.\ 1.3]{mavraki-schmidt} prove a uniform bound along any one-parameter family of pairs $(f,g)$ defined over $\Qbar$. More recently, DeMarco--Mavraki \cite{DM23} proves the existence of a Zariski open $U \subseteq \Rat_d \times \Rat_d$, where $\Rat_d$ is the space of all degree $d$ rational maps, for which a uniform bound holds.
\par 
Following Le Boudec--Mavraki \cite{LM21} and also Olechnowicz \cite{Ole23}, we will take a statistical point of view and instead consider the average count of common preperiodic points over pairs of monic polynomials over $\bb{Q}$. We consider only monic polynomials because the exceptional locus $\Prep(f) = \Prep(g)$ is easily described as $\{f = g\}$. We let 
\begin{align*}
\mc P&=\{z^d+a_{d-1}z^{d-1}+\cdots+a_0\,|\,a_i\in\bb Q \},\\
\mc P_c&=\{z^d+a_{d-1}z^{d-1}+\cdots+a_0\in\mc P\,|\,a_{d-1}=0\},
\end{align*}
be the set of monic polynomials and \emph{monic, centered polynomials} respectively with rational coefficients. For $X\geq1$, define
\begin{align*}
\mc P(X)&=\{z^d+a_{d-1}z^{d-1}+\cdots+a_0\,|\,a_i\in\bb Q,\,H(a_i)\leq X\},\\
\mc P_c(X)&=\{z^d+a_{d-1}z^{d-1}+\cdots+a_0\in\mc P(X)\,|\,a_{d-1}=0\},
\end{align*}
the set of monic polynomials and monic, centered polynomials respectively with rational coefficients of Weil height $\leq X$ (see Section \ref{sec:Def}). Note that any pair of monic polynomials $(f,g)$ over $\bb{Q}$ can be simultaneously conjugated by a translation to lie in $\mc P(X) \times \mc P_c(X)$ for some $X$.

In \cref{sec:prep}, we prove:
\begin{restatable}{theorem}{mainthmprep}\label{main-thm-prep}
Let $d\geq6$, and write $\mc S(X)=\PxP\setminus\{(f,f)\,|\,f\in\mc P_c(X)\}$. Then
\[\lim_{X\to\infty}\frac1{|\mc S(X)|}\sum_{(f,g)\in\mc S(X)}|\Prep(f)\cap\Prep(g)|=0\]
where $\Prep(f),\Prep(g)$ denotes all preperiodic points of $f$ and $g$ in $\ovl{\bb{Q}}$. 
\end{restatable}

\begin{remark}
The condition $d \geq 6$ is purely technical and we certainly expect the average to be zero for all degrees $d \geq 2$. With more work, we believe it is possible to improve our lower bound on $d$. Also, many of our lemmas adapt straightforwardly to number fields $K$ instead of $\bb{Q}$ (see Section \ref{sec:energy-upper}), but it is not clear if we can obtain a version of Theorem \ref{main-thm-prep} where we vary over polynomials over a fixed number field $K$ instead of $\bb{Q}$.
\end{remark}

Our arguments also allow us to conclude the average remain zero even when restricted to coordinate slices of $\cal{P}$ of large enough dimension. Let $m$ be a positive integer and let $V$ be a linear subspace of $\cal{P}$ of dimension $m$ that can be written as an intersection of some coordinate hyperplanes $\{a_i = 0\}$. Let $x = (x_{d-1},\ldots,x_0) \in \bb{Q}^{d-1}$ and let $A$ be the affine subspace given by $x + V$. We shall assume that the projection of $V$ onto the last coordinate $a_0$ is surjective. Choose $m$ coordinates $a_{i_1},\ldots,a_{i_m}$ that give rise to an isomorphism $V \simeq \bb{Q}^m$ such that $i_j$ are decreasing and $i_m = 0$. We define
$$\cal{P}_{A}(X) = \{z^d + a_{d-1} z^{d-1} + \cdots + a_0 \mid (a_{d-1},\ldots,a_0) \in A \text{ and } H(a_{i_j}) \leq X) \},$$
$$\cal{P}_{c,A}(X) = \{z^d + a_{d-2} z^{d-2} + \cdots + a_0 \mid (0,a_{d-2},\ldots,a_0) \in A \text{ and } H(a_{i_j}) \leq X\},$$
 $$\cal{S}_{A}(X) = \cal{P}_{c,A}(X) \times \cal{P}_{A}(X) \setminus \{(f,f) \mid f \in \cal{P}_{c,A}(X) \}.$$
\begin{restatable}{theorem}{mainthmprepsubspace} \label{main-thm-prep-subspace}
Assume that $\dim_{\bb{Q}} V \geq \frac{d}{2} + 3$. Then
$$\lim_{X \to \infty} \frac{1}{|\cal{S}_A(X)|} \sum_{(f,g) \in \cal{S}_A(X)} |\Prep(f) \cap \Prep(g)| = 0.$$
 \end{restatable}
\par 
In another direction, when $f,g$ are defined over $\Qbar$, we can reinterpret the common preperiodic points of $f,g$ as the set of common zeroes of the Call--Silverman canonical heights of $f,g$ \cite{call-silverman}:
\[\Prep(f)\cap\Prep(g)=\{x\in\Qbar\,|\,\widehat h_f(x)+\widehat h_g(x)=0\}.\]
In this setting, Baker--DeMarco and Mimar prove the following stronger form of \cref{baker-demarco}:
\begin{theorem}[{\cites{mimar}[Thm.\ 1.8]{Mimar2}}]
For any pair of rational maps $f,g:\bb{P}^1 \to \bb{P}^1$ of degree $d\geq2$ defined over $\ovl{\bb{Q}}$, we have
\[\liminf_{x\in\Qbar}\left(\widehat h_f(x)+\widehat h_g(x)\right)>0\quad\text{or}\quad\widehat h_f=\widehat h_g.\]
\end{theorem}
Here $\liminf_{x\in\Qbar}F(x)$ refers to the infimum of all accumulation points in $\bb{R}$ of the set $\{F(x)\,|\,x\in\Qbar\}$; equivalently,
\[\liminf_{x\in\Qbar}F(x)=\sup_{\substack{S\subseteq\Qbar\\|S|<\infty}}\inf_{x\in\Qbar\setminus S}F(x).\]
In particular if $\widehat h_f \not = \widehat h_g$, then there exists some $\eps=\eps(f,g)>0$ such that $\widehat h_f(x)+\widehat h_g(x)\leq\eps$ has only finitely many solutions $x\in\Qbar$; i.e. there are only finitely many points $x \in \ovl{\bb{Q}}$ that are both small for both $\widehat{h}_f$ and $\widehat{h}_g$. A natural question to ask is how large can $\eps$ be for a generic pair $(f,g)$ as the height of $f$ and $g$ grows. For any fixed $\delta > 0$, applying Zhang's inequality \cite[Thm.\ 1.10]{zhang95} (see Section \ref{sec:bogomolov}) gives us 
$$\left(\frac{1}{2} - \delta \right) (h(f) + h(g)) \leq d \liminf_{x \in \ovl{\bb{Q}}} \left(\widehat h_f(x)+\widehat h_g(x)\right) \leq (1 + \delta) (h(f) + h(g)) $$
for a generic pair $(f,g) \in \PxP$ and $X$ large enough, where $h(f)$ is the height of a polynomial $f$ (see Section \ref{sec:Def}). Our next theorem improves this inequality. 

\begin{restatable}{theorem}{mainthmbogomolov}\label{main-thm-bogomolov}
Fix a $\delta > 0$. We let $\cal{S} = \cal{P}_c \times \cal{P}$ and $\cal{S}(X) = \cal{P}_c(X) \times \cal{P}(X)$. Then for all sufficiently large degrees $d$, there exists $a = a(d) > 0, A = A(d) > 0$ such that the following statement holds: there exists a subset $\cal{T} \subseteq \cal{S}$ such that for all $X \geq 1$,
$$\frac{|T \cap \cal{S}(X)|}{|\cal{S}(X)|} \geq 1 - \frac{A}{X^a}$$
and
\[\left(\ln 2-\delta\right)(h(f)+h(g)) \leq d \liminf_{x\in\ovl{\bb Q}}\left(\widehat h_f(x)+\widehat h_g(x)\right)\leq\left(\mc C+\delta\right)(h(f)+h(g)),\]
for all $(f,g) \in \cal{T}$, where
$\ln 2 = 0.69314 \ldots$ and $\mc C=\log\left(\frac{9+\sqrt{17}}8\right)+\frac{\sqrt{17}-1}8=0.88532\ldots$
\end{restatable}

It will be interesting to determine the exact asymptotic, if it exists at all.


\subsection{Strategy of proof}
Our main tool for the proof of Theorems \ref{main-thm-prep} and \ref{main-thm-prep-subspace} is the energy pairing of Favre--Rivera-Letelier in \cite{FRL06}, see also \cite{PST12}. Given a polynomial $f: \bb{P}^1 \to \bb{P}^1$ defined over a number field $K$ of degree $d \geq 2$, for each place $v \in M_K$ we can attach to it a corresponding equilibrium measure $\mu_{f,v}$. For archimedean $v$, this is a probability measure on $\bb{C}$ while for non-archimedean $v$, it is a probability measure on the Berkovich affine line $\berkA$. The collection of measures $\{\mu_{f,v}\}_{v \in M_K}$ gives us an adelic measure, where $M_K$ are the places of $K$. For further details on $\berkA$, one may consult \cite{BR10} or \cite{Ben}.
\par 
Given two polynomials $f,g$, Favre--Rivera-Letelier first define a local energy pairing between $\mu_{f,v}$ and $\mu_{g,v}$ for a place $v$, and then define a global energy pairing by summing the local pairings over $v \in M_K$. This gives us a quantitative way to measure a distance between $f$ and $g$. Using the energy pairing and quantitative equidistribution, DeMarco--Krieger--Ye were able to establish a uniform bound on the number of common preperiodic points for quadratic unicritical polynomials.
\par 
We follow their approach in using the adelic energy pairing, proving bounds that hold for all monic polynomials in general but are weaker as we do not need a uniform bound. The hardest situation to handle is when both polynomials are close to each other. For example when $f(z)$ has integer coefficients and $g(z) = f(z)+1$, the Julia sets at every non-archimedean place are the corresponding Gauss point $\zeta(0,1)$ and so their local energy pairing is zero. But if $f$ and $g$ have large coefficients, then they are also relatively close for the archimedean places. To obtain a usable bound for Theorems \ref{main-thm-prep} and \ref{main-thm-prep-subspace}, we use the fact that the moments of the equilibrium measure $\mu_{f,v}$ are determined by the coefficients of $f$ when $f$ is monic (see Section \ref{subsec:energy-arch-lower}).
\par 
Now for \cref{main-thm-bogomolov}, we introduce a notion of genericity for pairs of polynomials in $\PxP$ which may be of independent interest. For an integer $n$, we let $\rad(n)$ be the product of all distinct primes that divides $n$. 

\begin{definition}\label{def:intro-eps-ord}
Given $X$ and $\eps>0$, we will say that a pair $(f,g)\in \PxP$ is \emph{$\eps$-ordinary} if
\begin{align*}
\rad(\denom(c))&\geq X^{1-2\eps},\\
\gcd(\denom(c),\denom(c'))&\leq X^{2\eps},
\end{align*}
for any distinct coefficients $c,c'$ of $f$ or $g$.
\end{definition}

In other words, a pair $(f,g) \in \PxP$ is $\eps$-ordinary if the denominators of the coefficients for both $f$ and $g$ are large, and are mostly coprime and squarefree. Similarly, we can make an analogous definition for a single polynomial $f \in \cal{P}(X)$ to be $\eps$-ordinary. As a consequence of the definitions, for most places $v$ of bad reduction of an $\eps$-ordinary pair, there is only one coefficient of $f,g$ that satisfies $|c|_v > 1$. 
\par 
The dynamics of an $\eps$-ordinary polynomial $f$ at such a place $v$ is simple enough to be analyzed thoroughly, but still more complicated than that of a unicritical polynomial. For example when $v$ is non-archimedean, elements $x$ of the Julia set of $f$ have three possible absolute values $|x|_v$( see \cref{NonArchimedean1}) whereas that for the Julia set of a polynomial of the form $z^d + c$ has one absolute value when $|c|_v > 1$. In Section \ref{sec:bogomolov}, we study the properties of the dynamical Green's function $g_{f,v}$ and equilibrium measure $\mu_{f,v}$.

For such pairs, we control some measures of arithmetic complexity and the structure of the Julia sets in \cref{sec:generic}. The lower bound of \cref{main-thm-bogomolov} then becomes a consequence of the structure and the product formula. The upper bound is more involved as it requires us to construct Galois orbits whose local heights are not too large. We do so by constructing an appropriate adelic set and then applying the adelic Fekete--Szeg\H{o} theorem due to Rumely \cite{rumely2014}.

\subsection{Acknowledgements} The authors would like to thank our advisor Laura DeMarco for helpful discussions and Niki Myrto Mavraki for suggesting the problem.


\section{Definitions and notations} \label{sec:Def}
Here we collect some notation used throughout this paper.

The cardinality of a set $S$ is written as $|S|$ or $\#S$.

We write $f=O(g)$ or $f\lesssim g$ for $f\leq Cg$ where $C>0$ is an absolute constant; $f=O_{\alpha,\beta,\ldots}(g)$ or $f\lesssim_{\alpha,\beta,\ldots}g$ for $f\leq C_{\alpha,\beta,\ldots}g$ where $C_{\alpha,\beta,\ldots}>0$ is a constant depending on parameters $\alpha,\beta,\ldots$; and $f=o(g)$ for $\frac fg\to0$. In such asymptotic notation, the limit to be taken will either be explicitly stated or clear from context.

The set of preperiodic points of $f$ is denoted by $\Prep(f)$.

The Weil height $H$ on $\bb Q$ is defined by $H(\frac ab)=\max(|a|,|b|)$ for $a,b\in\bb Z$ with $b\neq0$ and $\gcd(a,b)=1$.

The radical of a positive integer $n$, denoted by $\rad(n)$, is defined as the product of all distinct primes that divide $n$.

For a number field $K$, the set of all places for $K$ is denoted by $M_K$.

For each finite place $v\mid p<\infty$, the non-archimedean norm $|\cdot|_v$ on $\bb C_v$ is normalized with $|p|_v=\frac1p$. For a polynomial $f\in\bb Q[z]$ given by $f(z)=z^d+a_{d-1}z^{d-1}+\cdots+a_0$, the height of $f$ is defined as
\[h(f)=h_{\bb P^d}([1:a_{d-1}:\cdots:a_0])=\sum_{v\in M_{\bb Q}}\log\max(1,|a_{d-1}|_v,\ldots,|a_0|_v).\]

We now recall some notions from non-archimedean dynamics; we refer the reader to \cite{Ben} for the theory of dynamics on Berkovich space, and to \cite{BR10} for non-archimedean potential theory.

For each finite place $v\mid p<\infty$, write $\berkA$ for the Berkovich affine line. Now $\bb C_v$ embeds in $\berkA$ as the type I points, which form a dense set under the weak/Gel'fand topology. Write $\delta(z,w)_\infty$ for the Hsia kernel on $\berkA$, the unique continuous extension of the metric $|z-w|_v$ on Type I points.



Write $\mu_{f,v}$ for the invariant measure of a polynomial $f\in\bb C_v[z]$, and $g_{f,v}$ for its potential function. When $f$ is monic, this potential is also equal to the escape rate function for $f$, i.e.,
\[g_{f,v}(\zeta)=\lim_{n\to\infty}d^{-n}\log^+[f^n]_\zeta.\]

The local energy pairing between two signed measures $\mu_v,\mu_v'$ on $\berkA$ is defined as
\[(\mu_v,\mu_v')_v=-\int_{\berkA\times\berkA\setminus\Diag}\!\log \delta(z,w)_\infty\,d\mu_v(z)\otimes d\mu'_v(w).\]
The energy of $\mu_v$ is defined as
\[I(\mu_v)=\int_{\berkA\times\berkA}\!\log \delta(z,w)_\infty\,d\mu_v(z)\otimes d\mu_v(w),\]
which is also equal to $-(\mu_v,\mu_v)_v$ when $\mu_v$ has continuous potential. By abuse of notation, for any compact set $E_v\subseteq\berkA$ we define the logarithmic capacity of $E_v$ as
\[I(E_v)=\sup_{\mu_v}I(\mu_v),\]
where the supremum is taken over all probability measures $\mu_v$ supported on $E_v$. When this is greater than $-\infty$, we have $I(E_v)=I(\mu_{E_v})$, where $\mu_{E_v}$ is the equilibrium measure of $E_v$.

An adelic measure is a collection of probability measures $\{\rho_v\}_{v \in M_{K}}$ on $\berkA$ such that $\rho_v$ has continuous potentials and $\rho_v = \delta_{\zeta(0,1)}$ for all but finitely many $v$. Following Favre--Rivera-Letelier \cite{FRL06}, we can define a global pairing between two adelic measures $\rho,\rho'$ by
\[\langle\rho,\rho'\rangle=\frac12\sum_{v\in M_{\bb Q}}(\rho_v-\rho'_v,\rho_v-\rho'_v)_v.\]


For a finite $\Gal(\Qbar/\bb Q)$-invariant set $F\subseteq\Qbar$, write $[F]$ for the uniform probability distribution on $F$, i.e.,
\[[F]\coloneqq\frac1{|F|}\sum_{w\in F}\delta_w.\]
Furthermore, for $\eps>0$ we can define the regularized measure
\[[F]_\eps\coloneqq\frac1{|F|}\sum_{w\in F}\delta_{w,\eps},\]
where $\delta_{w,\eps}$ is the uniform probability measure on $\del B(w,\eps)$ for $v$ archimedean, or the delta mass on $\zeta(w,\eps)$ for $v$ non-archimedean.

For any polynomial $f\in\bb Q[z]$, we will generalize the definition of the canonical height $\widehat h_f$ to finite $\Gal(\Qbar/\bb Q)$-invariant sets $F\subseteq\Qbar$ by
\[\widehat h_f(F)\coloneqq\frac1{|F|}\sum_{w\in F}\widehat h_f(w).\]
Note that since $\widehat h_f$ is constant on $\Gal(\Qbar/\bb Q)$-orbits, we have $\widehat h_f(x)=\widehat h_f(F)$ when $F$ is the set of Galois conjugates of $x\in\Qbar$.

For a fixed degree $d$, we have the following sets
\begin{align*}
\mc P&=\{z^d+a_{d-1}z^{d-1}+\cdots+a_0\,|\,a_i\in\bb Q \},\\
\mc P_c&=\{z^d+a_{d-1}z^{d-1}+\cdots+a_0\in\mc P\,|\,a_{d-1}=0\},
\end{align*}
of monic polynomials with rational coefficients and monic, centered polynomials with rational coefficients respectively. For $X \geq 1$, we also define
\begin{align*}
\mc P(X)&=\{z^d+a_{d-1}z^{d-1}+\cdots+a_0\,|\,a_i\in\bb Q,\,H(a_i)\leq X\},\\
\mc P_c(X)&=\{z^d+a_{d-1}z^{d-1}+\cdots+a_0\in\mc P(X)\,|\,a_{d-1}=0\},
\end{align*}
as subsets of $\mc P$ and $\mc P_c$ consisting of polynomials with coefficients all having Weil height $\leq X$.

\section{Upper bound on energy pairing} \label{sec:energy-upper}
Let $f,g$ be two monic polynomials of degree $d \geq 2$ defined over $\bb{Q}$ and let $\mu_f = \{\mu_{f,v}\}_{v \in M_{\bb{Q}}}, \mu_g = \{\mu_{g,v}\}_{v \in M_{\bb{Q}}}$ be the adelic measures associated to them. Here, $\mu_{f,v},\mu_{g,v}$ denotes the equilibrium measures of $f$ and $g$ at the place $v$ respectively and they are both probability measures supported on $\berkA$.
\par 
Recall from the proof outline for \cref{main-thm-prep} that we want to bound the adelic energy pairing 
$$\langle\mu_f,\mu_g\rangle = \frac{1}{2} \sum_{v \in M_K} \int_{\berkA\times\berkA\setminus\Diag}\!\log \delta(z,w)_\infty\,d\mu_v(z)\otimes d\mu'_v(w)$$ 
on both sides. In this section, we give an upper bound in terms of 
$|\Prep(f)\cap\Prep(g)|$ and $h(f),h(g)$, following the approach of DeMarco--Krieger--Ye \cite{DKY22}. First, the energy pairing $\langle\cdot,\cdot\rangle^{1/2}$ is a metric on adelic measures (\cref{fili-dist}). Now by the Favre--Rivera-Letelier approach to quantitative equidistribution \cite{FRL06}, if $F\subseteq \Prep(f)\cap\Prep(g)$ is Galois-invariant and large, then for some suitable $\epsilon > 0$, the adelic measure $[F]_\eps$ must be close to both $\mu_f$ and $\mu_g$ in this metric; hence $\mu_f$ and $\mu_g$ must be close to each other, i.e., $\langle\mu_f,\mu_g\rangle$ must be small.

As input to this argument, we prove technical estimates on the H\"older continuity of the equilibrium potential $G_{f,v}$. We then finish the argument using quantitative equidistribution. 

\subsection{Archimedean bounds} \label{subsec:energy-upper-arch}
We will first handle the case where $v$ is an archimedean place. As all computations here are local, we will omit $v$ from our notation. 

We give a uniform H\"older bound on the equilibrium potential $G_f$ (\cref{holder-unif-arch}), by an argument of Carleson--Gamelin \cite[Thm.\ VIII.3.2]{carleson-gamelin}. We begin by bounding $G_f$ near the Julia set.
\begin{proposition}\label{PotentialBoundArch}
Let $\Omega=\{z\in\bb C\,:\,\dist(z,K_f)\leq1\}$. Suppose $f$ has Lipschitz constant $A$ on $\Omega$, i.e.,
\[\forall z,z'\in\Omega\,[|f(z)-f(z')| \leq A|z-z'|].\]
Also suppose that $G_f\leq M$ on $\Omega$ for some $M>0$. Then for all $z\in\Omega$, we have
\[G_f(z)\leq dM\dist(z,K_f)^\alpha,\]
where $\alpha=\frac{\log d}{\log A}$.
\end{proposition}

\begin{proof}
If $z\in K_f$, then both sides of the claimed inequality are 0.

Suppose $z\in \Omega\setminus K_f$. Then $z$ is in the attracting basin of $\infty$, so there exists a minimal $n\geq1$ for which $f^n(z)\not\in\Omega$. Since $K_f$ is invariant under $f$, the Lipschitz condition on $f$ gives
\[\dist(f(z),K_f)\leq A\dist(z,K_f).\]
Now $f^n(z)\not\in\Omega$ implies $\dist(z,K_f)\geq A^{-n}$, and $f^{n-1}(z)\in\Omega$ implies
\[M\geq G_f(f^{n-1}(z))=d^{n-1}G_f(z).\]
Hence
\[G_f(z)\leq dMd^{-n}=dM(A^{-n})^\alpha\leq dM\dist(z,K_f)^\alpha\]
giving us our upper bound on $G_f(z)$. 
\end{proof}

Now we give explicit values for the various parameters in \cref{PotentialBoundArch}. First we bound the size of the filled Julia set $K_f$ in terms of the coefficients of $f$.

\begin{lemma}\label{JuliaSetBound1Arch}
Let $R_f\coloneqq3\max(1,|a_{d-1}|,|a_{d-2}|^{1/2},\ldots,|a_0|^{1/d})$. Then $K_f\subseteq B(0,R_f)$.
\end{lemma}
\begin{proof}
For $|z|\geq R_f$, we have for $i\geq1$ that $|a_{d-i}z^{d-i}|\leq3^{-i} |z|^d$. Hence
\[|f(z)|\geq|z^d|\left(1-\frac{1}{3}-\frac{1}{3^2}-\cdots-\frac{1}{3^d}\right)\geq\frac{1}{2} |z^d|\geq\frac{3}{2} |z|.\]
Hence $|f^n(z)|\geq\left(\frac{3}{2}\right)^n|z|$, so $z$ lies in the attracting basin of $\infty$, i.e., $z\not\in K_f$.
\end{proof}

\begin{proposition}\label{JuliaSetConstsArch}
In \cref{PotentialBoundArch}, we may take
\[M=\log(2R_f+1),\qquad A=\frac{3d}2(R_f+1)^{d-1},\]
with $R_f$ as in \cref{JuliaSetBound1Arch}.
\end{proposition}
\begin{proof}
Note that $K_f\subseteq B(0,R_f)$ implies $\Omega\subseteq B(0,R_f+1)$, so $|z-w| \leq2R_f+1$ for all $z\in\Omega$ and $w\in K_f$. Now $\supp\mu_f=K_f$, so
\[G_f(z)=\int\log |z-w| d \mu_f(w)\leq \log(2R_f+1),\]
which proves the first claim.

Next, the Lipschitz constant of $f$ on $\Omega$ is at most $\sup_\Omega |f'|$. Now for $z\in\Omega\subseteq B(0,R_f+1)$, we have
\begin{align*}
|f'(z)| &\leq d|z^{d-1}|+(d-1)|a_{d-1}z^{d-2}|+\cdots+|a_1|\\
&\leq d\left(1+\frac{1}{3}+\frac{1}{3^2}\cdots\right)(R_f+1)^{d-1}\\
&=\frac{3d}{2}(R_f+1)^{d-1},
\end{align*}
which proves the second claim.
\end{proof}

Since the equilibrium potential is asymptotic to $\log |\cdot|$ far from the origin, the H\"older bound for $G_f$ near the Julia set (\cref{PotentialBoundArch}) yields a global H\"older bound. In fact, we can take the same constants as in \cref{JuliaSetConstsArch}:
\begin{proposition}\label{PotentialBoundAllArch}
Let $M,A$ be as in \cref{JuliaSetConstsArch}, and let $\alpha=\frac{\log d}{\log A}$. For all $z\in\bb C$, we have
\[G_f(z)\leq dM\dist(z,K_f)^\alpha.\]
\end{proposition}
\begin{proof}
Let $x=\dist(z,K_f)$. If $x\leq1$, we are done by \cref{PotentialBoundArch,JuliaSetConstsArch}.

Suppose that $x\geq1$. We have $K_f\subseteq B(0,R_f)$, so the diameter of $K_f$ is at most $2R_f$. Hence for every $w\in K_f$ we have $|z-w|\leq x+2R_f$, so
\[G_f(z)=\int\log|z-w| d\mu_f(w)\leq\log(x+2R_f).\]
Thus it suffices to show that $\log(x+2R_f) \leq dMx^\alpha$. 

\emph{Case 1}: $1\leq x\leq A$. Now $R_f\geq3$, so
$$
e^{dM}=(2R_f+1)^d \geq7(2R_f+1)^{d-1} \geq6\left(\frac74\right)^{d-1}(R_f+1)^{d-1}+2R_f \geq\frac{3d}2(R_f+1)^{d-1}+2R_f =A+2R_f.
$$
Hence
\begin{equation} \label{eq:PotentialBoundAllArch-1}
\log(x+2R_f)\leq\log(A+2R_f)\leq dM\leq dMx^\alpha.
\end{equation}

\emph{Case 2}: $x\geq A$. Write $x=A^c$ with $c\geq1$; now $x\geq A\geq2R_f\geq2$ , so
$$
\log(x+2R_f)\leq c\log A+\log2 \leq(c+1)\log A \leq d^c\log A \leq dMd^c=dMx^\alpha,
$$
where the last inequality follows from $\log A\leq dM$, by (\ref{eq:PotentialBoundAllArch-1}).
\end{proof}

We can now show uniform H\"older continuity for $G_f$.
\begin{proposition}\label{holder-unif-arch}
Let $M,A$ be as in \cref{JuliaSetConstsArch}, and let $\alpha=\frac{\log d}{\log A}$. For all $z_1,z_2\in\bb C$, we have
\[|G_f(z_1)-G_f(z_2)|\leq3dM |z_1-z_2|^\alpha.\]
\end{proposition}
\begin{proof}
\cite[Thm.\ VIII.3.2]{carleson-gamelin} Write $r_i=\dist(z_i,K_f)$; we may assume that $r_1\geq r_2$. Note that $A=\frac{3d}2(R_f+1)^{d-1}\geq d$, so that $\alpha\leq1$.

\emph{Case 1}: $|z_1-z_2| \geq\frac{1}{3r_1}$. Now $0\leq G_f(z_i)\leq dMr_i^\alpha$ by \cref{PotentialBoundAllArch}, so
\[|G_f(z_1)-G_f(z_2)|\leq dMr_1^\alpha\leq3dM|z_1-z_2|^\alpha.\]

\emph{Case 2}: $|z_1-z_2|=cr_1$, with $c\leq\frac13$. Now $G_f$ is harmonic and non-negative on $B(z_1,r_1)$, so by Harnack's inequality we have
\[\frac{1-c}{1+c}G_f(z_1)\leq G_f(z_2)\leq\frac{1+c}{1-c}G_f(z_1).\]
Therefore
$$
|G_f(z_1)-G_f(z_2)| \leq\frac{2c}{1-c}G_f(z_1) \leq3cdMr_1^\alpha \leq3dM|z_1-z_2|^\alpha.
$$
as desired.
\end{proof}

\begin{remark}
The arguments in \cref{PotentialBoundArch,holder-unif-arch} are essentially contained in \cite[Thm.\ VIII.3.2]{carleson-gamelin}; however, we note that their argument is incomplete, since they do not control the equilibrium potential for large $z$ as in our proof of \cref{PotentialBoundAllArch}.
\end{remark}

The H\"older bound on $G_f$ allows us to control the local pairing of $\mu_f$ with a regularized measure:
\begin{corollary}\label{holder-arch}
Let $M,A$ be as in \cref{JuliaSetConstsArch}, and let $\alpha=\frac{\log d}{\log A}$. For any $\eps>0$ and any nonempty finite set $F\subseteq\bb C$, we have
\[-(\mu_f,[F]_\eps)\leq-(\mu_f,[F])+3dM\eps^\alpha.\]
\end{corollary}
\begin{proof}
Since $|w-\xi|=\eps$ for all $\xi\in \del B(w,\eps)$, we have
$$
-(\mu_f,[F]_\eps) =\frac1{|F|}\sum_{w\in F}\int_{\bb C}G_f(\xi)d\delta_{w,\eps}(\xi) \leq\frac1{|F|}\sum_{w\in F}(G_f(w)+3dM\eps^\alpha) =-(\mu_f,[F])+3dM\eps^\alpha
$$
as desired.
\end{proof}

\subsection{Non-archimedean bounds} \label{subsec:energy-upper-nonarch}
Now we consider the non-archimedean case $v\mid p<\infty$ with $v\in M_K$. By choosing an appropriate embedding $K\hookrightarrow\bb C_p$, we can view $|\cdot|_v=|\cdot|_p$.

We first show that the Lipschitz constant of $f$ on a closed subset of $\berkA$ is the same as the Lipschitz constant on its Type I points.
\begin{lemma}\label{LipschitzTypeI}
Let $\Omega\subseteq\berkA$ be a closed subset. Suppose there is a constant $A>0$ such that for all $z,z'\in\Omega\cap\bb C_v$, we have
\[|f(z)-f(z')|_v\leq A |z-z'|_v.\]
Then for all $\zeta,\zeta'\in\Omega$, we have
\[\delta(f(\zeta),f(\zeta'))_\infty\leq A\delta(\zeta,\zeta')_\infty.\]
\end{lemma}

\begin{proof}
\emph{Case 1}: $f(\zeta)\neq f(\zeta')$. Take sequences $(z_n),(z'_n)\subseteq\bb C_v$ such that $(z_n)\to\zeta$ and $(z'_n)\to\zeta'$. Since $\delta(\cdot,\cdot)_\infty$ is continuous off the diagonal, we have
\begin{align*}
\delta(f(\zeta),f(\zeta'))_\infty&=\lim_{n\to\infty}|f(z_n)-f(z'_n)|_v\\
&\leq A\lim_{n\to\infty}|z_n-z'_n|_v\\
&=A\delta(\zeta,\zeta')_\infty.
\end{align*}

\emph{Case 2}: $f(\zeta)=f(\zeta')$. Write $\zeta=\zeta(a,r)$ and $f(\zeta)=\zeta(b,R)$; then \cite[Prop.\ 7.6]{Ben} gives $f(B(a,r))=B(b,R)$, so
\[\delta(f(\zeta),f(\zeta'))_\infty=R\leq Ar\leq A\delta(\zeta,\zeta')_\infty.\qedhere\]    
\end{proof}

From here on, the structure of the argument for the non-archimedean case is analogous to the archimedean case in the previous section.

\begin{proposition}\label{PotentialBoundNonArch}
Let $\Omega_v=\{z\in\bb C_v\,:\,\dist_v(z,K_{f,v})\leq1\}$. Suppose $f$ has Lipschitz constant $A_v>1$ on $\Omega_v$, i.e.,
\[\forall z,z'\in\Omega_v\,[|f(z)-f(z')|_v\leq A_v|z-z'|_v].\]
Also suppose that $G_{f,v}\leq M_v$ on $\Omega_v$ for some $M_v>0$. Then for all $z\in\Omega_v$, we have
\[G_{f,v}(z)\leq dM_v\dist_v(z,K_{f,v})^{\alpha_v},\]
where $\alpha_v=\frac{\log d}{\log A_v}$.
\end{proposition}
\begin{proof}
In view of \cref{LipschitzTypeI}, we may replace $\Omega_v$ with its closure $\overline{\Omega_v}$ in $\berkA$. The proof is now completely analogous to the proof of \cref{PotentialBoundArch}.
\end{proof}

\begin{proposition}\label{JuliaSetBound1NonArch}
Let $R_{f,v} :=\max(1,|a_{d-1}|_v,|a_{d-2}|_v^{1/2},\ldots,|a_0|_v^{1/d})$. Then $K_{f,v} \subseteq\overline B(0,R_{f,v})$.
\end{proposition}
\begin{proof}
For $|z|_v>R_{f,v}$ and $i\geq1$, we have $|a_{d-i}z^{d-i}|_v<|z|_v^d$. Hence $|f(z)|_v=|z|_v^d$. By iteration we see that $|f^n(z)|_v=|z|_v^{d^n}$, so $z$ lies in the attracting basin of $\infty$, i.e., $z\not\in K_{f,v}$.
\end{proof}

\begin{proposition}\label{JuliaSetConstsNonArch}
Suppose that $f$ has bad reduction at $v$. In \cref{PotentialBoundNonArch}, we may take
\[M_v=\log R_{f,v},\qquad A_v=R_{f,v}^{d-1},\]
with $R_{f,v}$ as in \cref{JuliaSetBound1NonArch}.
\end{proposition}
\begin{proof}
First, $f$ having bad reduction at $v$ implies $R_{f,v}>1$. Now $K_{f,v}\subseteq\overline B(0,R_{f,v})$ implies that $\Omega_v\subseteq\overline B(0,R_{f,v})$ as well, so $|z-w|_v\leq R_{f,v}$ for all $z\in\Omega_v$ and $w\in K_{f,v}$. Now $\supp\mu_{f,v}=K_{f,v}$, so
\[G_{f,v}(z)=\int\log |z-w| d\mu_{f,v}(w)\leq\log R_{f,v},\]
which proves the first claim.

Next, note that for all $z\in\Omega_v\subseteq\overline B(0,R_{f,v})$ we have
\[|f(z)|\leq\max_i |a_{d-i}z^{d-i}|_v\leq R_{f,v}^d.\]
In other words, we have $f(\overline B(0,R_{f,v}))\subseteq\overline B(0,R_{f,v}^d)$. It now follows from \cite[Prop.\ 3.20]{Ben} that the Lipschitz constant for $f$ on $B(0,R_{f,v})$ is bounded above by $R_{f,v}^{d-1}$. Thus the same also holds on $\Omega_v$, which proves the second claim.
\end{proof}

\begin{remark}
When $f$ has explicit good reduction at $v$, we have $R_v=1$, so $A_v=1$, thus $\alpha_v=\frac{\log d}{\log A_v}$ is not well-defined, and \cref{PotentialBoundNonArch} cannot be applied. However, in this case we have $J_{f,v}=\{\zeta(0,1)\}$, so $\Omega_v=K_{f,v}\cap\bb C_v=\overline B(0,1)$, and thus $G_{f,v}\equiv0$ on $\Omega_v$.
\end{remark}

From now on in this subsection, we will assume that $f$ has bad reduction at $v$. 

\begin{proposition}\label{PotentialBoundAllNonArch}
Let $v,M_v,A_v$ be as in \cref{JuliaSetConstsNonArch}, and let $\alpha_v=\frac{\log d}{\log A_v}$. For all $z\in\bb C_v$, we have
\[G_{f,v}(z)\leq dM\dist(z,K_{f,v})^{\alpha_v}.\]
\end{proposition}
\begin{proof}
Let $x=\dist(z,K_{f,v})$. If $x\leq1$, we are done by \cref{PotentialBoundArch,JuliaSetConstsNonArch}.

Suppose that $x\geq1$. We have $K_{f,v}\subseteq B(0,R_{f,v})$, so the diameter of $K_{f,v}$ is at most $R_{f,v}$. Hence for every $w\in K_{f,v}$ we have $\delta(z,w)_\infty\leq\max(x,R_{f,v})$, so
\[G_{f,v}(z)=\int\log\delta(z,w)_\infty d \mu_{f,v}(w)\leq\log\max(x,R_{f,v}).\]
Thus it suffices to show that $\log\max(x,R_{f,v}) \leq dM_vx^{\alpha_v}$.

Now $\log R_{f,v}=M_v\leq dM_vx^{\alpha_v}$; also, writing $x=A_v^c$ for some $c\geq0$, we have
$$
\log x c(d-1)\log R_{f,v} \leq d^{c+1}\log R_{f,v} =
 dM_vd^c=dM_vx^{\alpha_v}.\qedhere
$$
\end{proof}

\begin{proposition}\label{holder-unif-nonarch}
Let $v,M_v,A_v$ be as in \cref{JuliaSetConstsNonArch}, and let $\alpha_v=\frac{\log d}{\log A_v}$. For all $z_1,z_2\in\berkA$, we have
\[|G_{f,v}(z_1)-G_{f,v}(z_2)|\leq dM_v|z_1-z_2|_v^{\alpha_v}.\]
\end{proposition}
\begin{proof}
Write $r_i=\dist_v(z_i,K_{f,v})$; we may assume that $r_1\geq r_2$.

\emph{Case 1}: $|z_1-z_2|_v \geq r_1$. Now $0\leq G_{f,v}(z_i)\leq dM_vr_i^{\alpha_v}$ by \cref{PotentialBoundNonArch}, so
\[|G_{f,v}(z_1)-G_{f,v}(z_2)|\leq dM_vr_1^{\alpha_v}\leq dM_v |z_1-z_2|^{\alpha_v}.\]

\emph{Case 2}: $|z_1-z_2|_v<r_1$. Then for all $w\in K_{f,v}=\supp\mu_{f,v}$, we have
\[\delta(z_1,w)_\infty\geq r_1> |z_1-z_2|,\]
so that $\delta(z_1,w)_\infty=\delta(z_2,w)_\infty$. Hence
$$
G_{f,v}(z_1) =\int\log\delta(z_1,w)_\infty d\mu_{f,v}(w) =\int\log\delta(z_2,w)_\infty d\mu_{f,v}(w)=G_{f,v}(z_2).\qedhere
$$
\end{proof}

\begin{corollary}\label{holder-nonarch}
Let $M_v,A_v$ be as in \cref{JuliaSetConstsNonArch}, and let $\alpha_v=\frac{\log d}{\log A_v}$. For any $\eps>0$ and any nonempty finite set $F\subseteq\bb C_v$, we have
\[-(\mu_{f,v},[F]_\eps)_v\leq -(\mu_{f,v},[F])_v+dM_v\eps^{\alpha_v}.\]
\end{corollary}
\begin{proof}
Since $\delta(\zeta(w,\eps),w)_\infty=\eps$, we have
$$
-(\mu_{f,v},[F]_\eps)_v =\frac1{|F|}\sum_{w\in F}G_{f,v}(\zeta(w,\eps)) \leq\frac1{|F|}\sum_{w\in F}(G_{f,v}(w)+dM_v\eps^{\alpha_v}) =-(\mu_{f,v},[F])_v+dM_v\eps^{\alpha_v}.\qedhere
$$
\end{proof}

\subsection{Quantitative equidistribution}
To summarize the discussion in \cref{subsec:energy-upper-arch,subsec:energy-upper-nonarch}, we have shown:

\begin{theorem}\label{holder-bound-all}
Let $\eps>0$. For any place $v\in M_K$ such that either $v\mid\infty$ or $f$ has bad reduction at $v$, and any nonempty finite set $F\subseteq\bb C_v$, we have
\[-(\mu_f,[F]_\eps)_v\leq-(\mu_f,[F])_v+3dM_v\eps^{\alpha_v},\]
with constants 
$\alpha_v=\frac{\log d}{\log A_v},$
$$
M_v =\begin{cases}\log(2R_{f,\infty}+1) \quad v\mid\infty,\\\log R_{f,v}, \quad v\nmid\infty,\end{cases} 
A_v =\begin{cases}\frac{3d}2(R_{f,\infty}+1)^{d-1} \quad v\mid\infty, \\R_{f,v}^{d-1} \quad v\nmid\infty,\end{cases}$$
$$
R_{f,v} =\begin{cases}3\max(1,|a_{d-1}|_\infty,|a_{d-2}|_\infty^{1/2},\ldots,|a_0|_\infty^{1/d}) \quad v\mid\infty,\\\max(1,|a_{d-1}|_v,|a_{d-2}|_v^{1/2},\ldots,|a_0|_v^{1/d}) \quad v\nmid\infty.\end{cases}
$$
\end{theorem}
\begin{proof}
This is immediate from \cref{holder-arch,holder-nonarch}.
\end{proof}

To convert these H\"older bounds into an upper bound for the energy pairing $\langle\mu_f,\mu_g\rangle$, we show the following quantitative equidistribution result: for any large Galois-invariant subset $F\subseteq\Prep(f)$, there is some small adelic radius $\eps$ such that the regularized adelic measure $[F]_\eps$ is close to $\mu_f$.

\begin{lemma}[{\cites[Lem.\ 4.11]{FRL06}[Lem.\ 12]{Fil17}}]\label{LemFRL}
For any place $v\in M_K$, any nonempty finite set $F\subseteq\bb C_v$, and any $\eps>0$, we have
\[([F]_\eps,[F]_\eps)_v\leq([F],[F])_v+\frac{\log(1/\eps)}{|F|}.\]
\end{lemma}

\begin{proposition}\label{quant-equid-bound}
With notation as in \cref{holder-bound-all}, let $F\subseteq K$ be a finite $\Gal(\Qbar/\bb Q)$-invariant set with $|F|=N\geq2$. Define the adelic radius $\eps=\{\eps_v\}_{v\in M_K}$ as
\[\eps_v=\begin{cases}
(dN)^{-1/\alpha_v}&\text{if }v\mid\infty\text{, or }f\text{ has bad red.\ at }v,\\
1&\text{if }f\text{ has explicit good red.\ at }v.
\end{cases}\]
Then
\[\langle\mu_f,[F]_\eps\rangle\leq\widehat h_f(F)+O\left(d\frac{\log N}N(h(f)+1)\right).\]
\end{proposition}

\begin{proof}
We extend an argument of \cite[Lem.\ 9.2]{DKY22}. Let $\mc B\subseteq M_K$ be the set of places $v$ such that either $v\mid\infty$, or $f$ has bad reduction at $v$. 

First, by using $(\mu_f,\mu_f)_v=0$, we expand the left side of the claim as
\begin{equation} \label{eq:quant-equid-bound0}
\langle\mu_f,[F]_\eps\rangle =\frac1{2[K:\bb Q]}\sum_{v\in M_K}\!n_v([F]_\eps-\mu_f,[F]_\eps-\mu_f)_v \end{equation}
$$=\frac1{[K:\bb Q]}\sum_{v\in M_K}\!n_v\left(\frac12([F]_\eps,[F]_\eps)_v-(\mu_f,[F]_\eps)_v\right).
$$
We bound the first term in the sum of \cref{eq:quant-equid-bound0} by \cref{LemFRL} and the product formula $\sum_vn_v([F],[F])_v=0$, to obtain
\begin{equation} \label{eq:quant-equid-bound0a}
\frac1{2[K:\bb Q]}\sum_{v\in M_K}\!n_v([F]_\eps,[F]_\eps)_v \leq\frac1{[K:\bb Q]}\sum_{v\in M_K}\!n_v\frac{\log(1/\eps_v)}{2N} =\frac1{[K:\bb Q]}\sum_{v\in\mc B}n_v\frac{1+\frac{\log N}{\log d}}{2N}\log A_v.
\end{equation}

Next, we bound the second term in the sum of (\ref{eq:quant-equid-bound0}). If $v\in\mc B$ then we can use \cref{holder-bound-all}. Suppose $v\in M_k\setminus\mc B$, i.e., $f$ has explicit good reduction at $v$. Then $\eps_v=1$, and $\mu_{f,v}=\delta_{\zeta(0,1)}$, so $G_{f,v}=\log^+|\cdot|_v$. Hence
$$
-(\mu_f,[F]_\eps)_v =\frac1{|F|}\sum_{w\in F}\log^+|\zeta(w,1)|_v =\frac1{|F|}\sum_{w\in F}\log^+\max(1,|w|_v) 
=\frac1{|F|}\sum_{w\in F}\log^+|w|_v=-(\mu_f,[F])_v.
$$
The above two cases yield
\begin{equation} \label{eq:quant-equid-bound0b}
\frac1{[K:\bb Q]}\sum_{v\in M_K}\!n_v(-(\mu_f,[F]_\eps)_v) \leq\frac1{[K:\bb Q]}\left(\sum_{v\in M_K}\!n_v(-(\mu_f,[F])_v)+\sum_{v\in\mc B}n_v(3dM_v\eps_v^{\alpha_v})\right) 
\end{equation}
$$=\widehat h_f(F)+\frac1{[K:\bb Q]}\sum_{v\in\mc B}n_v(3dM_v\eps_v^{\alpha_v})
 =\widehat h_f(F)+\frac1{[K:\bb Q]}\sum_{v\in\mc B}n_v\frac{3M_v}N.
$$

Now (\ref{eq:quant-equid-bound0}), (\ref{eq:quant-equid-bound0a}), (\ref{eq:quant-equid-bound0b}), imply
\begin{equation} \label{eq:quant-equid-bound1}
\langle\mu_f,[F]_\eps\rangle-\widehat h_f(F) \lesssim\frac{\log N}N\frac1{[K:\bb Q]}\sum_{v\in\mc B}n_v(M_v+\log A_v) \leq\frac{\log N}N\sum_{v\in M_{\bb Q}}(M_v+\log A_v).
\end{equation}

We now bound the sum in (\ref{eq:quant-equid-bound1}). Note that
\[\sum_{v\in M_{\bb Q}}\log R_{f,v}=\sum_{v\nmid\infty}\log R_{f,v}+\log R_{f,\infty}
\leq h(f)+O(1).\]
For the first term $M_v$ in \cref{eq:quant-equid-bound1}, we have
\begin{equation} \label{eq:quant-equid-bound2}
\sum_{v\in M_{\bb Q}}M_v \leq\sum_{v\nmid\infty}\log R_{f,v}+\log(2R_{f,\infty}+1)
\leq\sum_{v\nmid\infty}\log R_{f,v}+\log R_{f,\infty}+O(1)
\leq h(f)+O(1).
\end{equation}
where the first inequality follows from \cref{holder-bound-all}.
Similarly, for the second term $\log A_v$ in (\ref{eq:quant-equid-bound1}), we deduce
\begin{equation} \label{eq:quant-equid-bound3}
\sum_{v\in M_{\bb Q}}\log A_v \leq\sum_{v\nmid\infty}(d-1)\log R_{f,v}+(d-1)\log(R_{f,\infty}+1)+\log(\tfrac32d) \leq(d-1)h(f)+O(d).
\end{equation}
Substituting (\ref{eq:quant-equid-bound2}),(\ref{eq:quant-equid-bound3}) back into (\ref{eq:quant-equid-bound1}), we deduce
\[\langle\mu_f,[F]_\eps\rangle-\widehat h_f(F)
\lesssim d\frac{\log N}N(h(f)+1).\qedhere\]
\end{proof}

We can now show that $\mu_f$ and $\mu_g$ are close to each other, thereby proving our upper bound for the adelic energy pairing.
\begin{lemma}[{\cite[Thm.\ 1]{Fil17}}]\label{fili-dist}
For any adelic measures $\rho,\sigma,\tau$, we have
\[\langle\rho,\tau\rangle^{1/2}\leq\langle\rho,\sigma\rangle^{1/2}+\langle\sigma,\tau\rangle^{1/2}.\]
\end{lemma}
\begin{theorem}\label{UpperBoundEnergyPairing}
Suppose $N=|\Prep(f)\cap\Prep(g)|\geq2$. Then
\[\langle\mu_f,\mu_g\rangle\lesssim d\frac{\log N}N(\max(h(f),h(g))+1).\]
\end{theorem}

\begin{proof}
Take $F=\Prep(f)\cap\Prep(g)$, so that $\widehat h_f(F)=\widehat h_g(F)=0$. Then by \cref{fili-dist,quant-equid-bound}, we have
$$
\langle\mu_f,\mu_g\rangle \leq\left(\langle\mu_f,[F]_\eps\rangle^{1/2}+\langle\mu_g,[F]_\eps\rangle^{1/2}\right)^2
\lesssim d\frac{\log N}N\left(\sqrt{h(f)+1}+\sqrt{h(g)+1}\right)^2$$
$$
\lesssim d\frac{\log N}N(\max(h(f),h(g))+1).\qedhere
$$
\end{proof}

\begin{remark}
The proof of \cref{UpperBoundEnergyPairing} in fact carries over to any pair of monic $f, g\in K[z]$, with the same implied constant independent of the field $K$.
\end{remark}

We will also use the quantitative equidistribution bound of \cref{quant-equid-bound} to give a $O(\log X)$ bound in some cases where we know the measures differ significantly in a non-archimedean place. 

\begin{lemma}\label{cs}
Let $v\in M_K$ be non-archimedean. Let $\varphi:\berkA\to\bb R$ be continuous with compact support, and let $\mu$ be a measure on $\berkA$ with total mass 0 and continuous potential. Then
\[ \left|\int_\berkA\varphi\dd\mu \right|\leq(\Lap\varphi,\Lap\varphi)_v^{1/2}(\mu,\mu)_v^{1/2}.\]
\end{lemma}
\begin{proof}
If the right side is $\infty$ then there is nothing to prove. Otherwise, let $\mu=\Lap\psi$. By \cite[Cor.\ 5.38]{BR10}, the expression $[F,G]:=\int_\berkA F\Lap G$ defines a positive semidefinite symmetric bilinear form on $\{F\in\operatorname{BDV}(\berkA)\,:\,[F,F]<\infty\}$. Now
$$
[\varphi,\psi]=\int_\berkA\varphi\dd\mu,\qquad
[\varphi,\varphi]=(\Lap\varphi,\Lap\varphi)_v, \qquad [\psi,\psi]=(\mu,\mu)_v,
$$
and we are done by the Cauchy--Schwarz inequality.

\end{proof}

\begin{proposition} \label{NonArchimedeanEnergyBound}
Let $v\in M_{\bb Q}$ be non-archimedean, with $p$ its associated prime. If
\[\mu_{g,v}(\berkA\setminus D_{\an}(0,p^{1/d}))\geq1/d,\]
then 
there are at most $C_d(h(g)^2+1)\log^2p$ points $x\in\Prep(g)$ such that $|\sigma(x)|_v\leq1$ for all $\sigma\in\Gal(\Qbar/\bb Q)$.
\end{proposition}

\begin{proof}
Consider the path segment $T$ that connects $\zeta(0,1)$ and $\zeta(0, p^{1/d})$ in $\berkA$, namely $T=\{\zeta(0,p^{t/d})\,:\,t\in[0,1]\}$. Let $\varphi$ be the continuous function on $\berkA$ which is locally constant on $\berkA\setminus T$, and takes values $\varphi(\zeta(0,p^{t/d}))=t$ on $T$. We check that
\[\Lap\varphi=\frac d{\log p}\left(\delta_{\zeta(0,1)}-\delta_{\zeta(0,p^{1/d})}\right).\]

Now if $F$ is a nonempty $\Gal(\Qbar/\bb Q)$-invariant set of preperiodic points for $g$ contained in $\overline B(0,1)$, take the adelic radius $\eps$ as in \cref{quant-equid-bound}. Since $0<\eps_v\leq1$, we have $\zeta(w,\eps_v)\in\ovl{D}_{\an}(0,1)$ for all $w\in F$. Hence
\[\int_\berkA\varphi d[F]_{\eps_v}=0,\qquad \int_\berkA\varphi\dd\mu_{g,v}\geq\frac1d.\]

Now \cref{quant-equid-bound} gives
$$
d\frac{\log N}N(h(g)+1)\gtrsim\langle\mu_g,[F]_\eps\rangle =\frac12\sum_{w\in M_{\bb Q}}(\mu_g-[F]_\eps,\mu_g-[F]_\eps)_w \geq\frac12(\mu_g-[F]_\eps,\mu_g-[F]_\eps)_v,$$
since the terms in the sum are non-negative. 
Then by \cref{cs},
$$
\frac1d\leq \left|\int\varphi\dd[F]_{\eps_v}-\int\varphi\dd\mu_{g,v} \right| \leq(\Lap\varphi,\Lap\varphi)^{1/2}_v([F]_{\eps_v}-\mu_{g,v},[F]_{\eps_v}-\mu_{g,v})^{1/2}_v \lesssim\left(\frac d{\log p}\right)^{1/2}\left(d\frac{\log N}N(h(g)+1)\right)^{1/2},
$$
which yields $\frac N{\log N}\lesssim d^4(h(g)+1)\log p$. Therefore $N\lesssim_d(h(g)^2+1)\log^2p$.
\end{proof}

\subsection{Lower bound on the archimedean energy pairing} \label{subsec:energy-arch-lower}

We now prove the lower bound on $\langle\mu_f,\mu_g\rangle$. In particular, we will give a lower bound on the local energy pairing $(\mu_f-\mu_g,\mu_f-\mu_g)_\infty$ at the archimedean place. As usual, we will suppress the subscript $\infty$ in this section. 

The simplest instance of our argument goes as follows: the centroid of $\mu_f$ is a constant multiple of $a_{d-1}$, and similarly for $\mu_g$. Thus if $a_{d-1}\neq b_{d-1}$, then $\mu_f$ and $\mu_g$ cannot be too close to each other, so we should get a lower bound on their energy pairing. In general, we will need to look at higher moments of $\mu_f,\mu_g$, depending on the first coefficient that $f,g$ differ in.

\begin{lemma}\label{MomentBound}
If for some $k\leq d$ we have $a_{d-j}=b_{d-j}$ for all $1\leq j<k$, then
\[\int z^k\dd\mu_f(z)-\int z^k\dd\mu_g(z)=-\frac kd(a_{d-k}-b_{d-k}).\]
\end{lemma}

\begin{proof}
\emph{Case 1}: $k<d$. For any $c\in\bb C$, let $\rho_1,\ldots,\rho_d$ be the roots of $f-c$. By Newton's identities for elementary symmetric polynomials, the power sum $\rho_1^k+\cdots+\rho_d^k$ is a polynomial in $a_{d-1},\ldots,a_{d-k}$, linear in $a_{d-k}$ with coefficient $-k$; in other words, there is some polynomial $P_k$ such that
\[\int z^k\dd(f^*\delta_c)(z)=-ka_{d-k}+P_k(a_{d-1},\ldots,a_{d-k+1}).\]

Now for any non-exceptional point $z_0\in\bb C$, we consider the iterated pullback $\mu_f^{(n)}=d^{-n}(f^n)^*\delta_{z_0}$. On one hand, we have $\mu_f^{(n)}=\frac1df^*\mu_f^{(n-1)}$, so
$$
\int z^k d \mu_f^{(n)}(z) =\frac1d\int z^kd(f^*\mu_f^{(n-1)})(z) =\frac1d\int\left(\int z^k d(f^*\delta_c)(z)\right) d\mu_f^{(n-1)}(c) $$
$$=-\frac kda_{d-k}+\frac1dP_k(a_{d-1},\ldots,a_{d-k+1}).
$$
On the other hand, the sequence $\mu_f^{(n)}$ has uniformly bounded support, and converges weakly to $\mu_f$ by equidistribution. Hence we also have
\[\int z^k d \mu_f(z)=-\frac kda_{d-k}+\frac1dP_k(a_{d-1},\ldots,a_{d-k+1}).\]
By replacing $f$ with $g$ we get an analogous equation, and by taking the difference we obtain the claim.

\emph{Case 2}: $k=d$. The constant coefficient of $f-c$ is $a_0-c$, so we now have
\[\int z^dd(f^*\delta_c)(z)=-d(a_0-c)+P_d(a_{d-1},\ldots,a_1)\]
for some polynomial $P_d$. Arguing as before, we obtain 
\[\int z^d d\mu_f(z)=-a_0+\int c\,d\mu_f(c)+\frac1dP_d(a_{d-1},\ldots,a_1).\]
By replacing $f$ with $g$ and taking the difference, we get
$$
\int z^d d \mu_f(z)-\int z^d d\mu_g(z) =-(a_0-b_0)+\left(\int z d \mu_f(z)-\int z d \mu_g(z)\right) =-(a_0-b_0),
$$
by the $k=1$ case of this lemma.
\end{proof}

\begin{proposition} \label{ArchimedeanLowerBoundEnergy}
Let $(f,g)\in\PxP$ with $f\neq g$. If $a_{d-j}=b_{d-j}$ for all $1\leq j<k$ and $a_{d-k}\neq b_{d-k}$, then
\[\langle\mu_f,\mu_g\rangle\gtrsim_d\begin{cases}X^{-4}&\text{if }k=1,\\X^{-k-4}&\text{if }2\leq k\leq d.\end{cases}\]
\end{proposition}

\begin{proof}
Recall from \cref{JuliaSetBound1Arch} that $J_f=\supp\mu_f\subseteq B(0,R_f)$, where
\[R_f=3\max(1,|a_{d-1}|,|a_{d-2}|^{1/2},\ldots,|a_0|^{1/d}).\]
\emph{Case 1}: $k\geq2$. Now $a_{d-1}=b_{d-1}=0$, so $R_f,R_g\leq3\sqrt X:= R$.

Fix a smooth function $\psi_0(z)$ which is zero outside of $B(0,2)$, and equal to $z^k$ in $B(0,1)$. Then the function $\psi(z)=R^k\psi_0(\frac zR)$ satisfies $\psi(z)=z^k$ in $B(0,R)$, and
\[||\nabla\psi||_{L^2}=R^k||\nabla\psi_0||_{L^2}.\]

Next, note that $a_{d-k},b_{d-k}$ are distinct rationals with numerator and denominator $\leq X$, so $|a_{d-k}-b_{d-k}| \geq X^{-2}$. Hence by \cref{MomentBound,cs}, we have
$$
\frac kdX^{-2}\leq\frac kd |a_{d-k}-b_{d-k}| = \left|\int\psi d\mu_f-\int\psi d \mu_g \right| $$
$$\leq ||\nabla\psi||_{L^2}(\mu_f-\mu_g,\mu_f-\mu_g)_\infty^{1/2} \leq(3\sqrt X)^k ||\nabla\psi_0||_{L^2}\langle\mu_f,\mu_g\rangle^{1/2}.
$$
Rearranging, we get
\[\langle\mu_f,\mu_g\rangle\geq\left(\frac{kX^{-k/2-2}}{3^kd ||\nabla\psi_0||_{L^2}}\right)^2\geq\frac{X^{-k-4}}{(3^dd ||\nabla\psi_0||_{L^2})^2}\gtrsim_dX^{-k-4}.\]

\emph{Case 2}: $k=1$. In this case, we can only bound $R_f,R_g$ from above by $R\coloneqq3X$; however, since $a_{d-1}=0\neq b_{d-1}$, we have the stronger bound $|a_{d-1}-b_{d-1}|\geq X^{-1}$. Now the same argument with these modifications yields 
\[\langle\mu_f,\mu_g\rangle\geq\left(\frac{X^{-1-1}}{3d ||\nabla\psi_0||_{L^2}}\right)^2\gtrsim_dX^{-4}.\qedhere\]
\end{proof}

\begin{corollary}\label{main-thm-prep-case3}
Let $(f,g)\in\PxP$ with $f\neq g$. If $a_{d-j}=b_{d-j}$ for all $1\leq j<k$ and $a_{d-k}\neq b_{d-k}$, then $|\Prep(f)\cap\Prep(g)| \lesssim_{\eps,d}X^{k+4+\eps}$.
\end{corollary}
\begin{proof}
Write $N=|\Prep(f)\cap\Prep(g)|$. By \cref{ArchimedeanLowerBoundEnergy,UpperBoundEnergyPairing},
\[X^{-k-4}\lesssim_d\langle\mu_f,\mu_g\rangle\lesssim d\frac{\log N}N(\max(h(f),h(g))+1).\]
Hence $\frac N{\log N}\lesssim_dX^{k+4}(\log X+1)\lesssim_\eps X^{k+4+\eps/2}$, thus $N\lesssim_{\eps,d}X^{k+4+\eps}$.
\end{proof}

\section{Average number of common preperiodic points}\label{sec:prep}
The main purpose of this section is to establish \cref{main-thm-prep}. Recall for $X \geq 1$, we have the sets 
\begin{align*}
\mc P(X)&=\{z^d+a_{d-1}z^{d-1}+\cdots+a_0\,|\,a_i\in\bb Q,\,H(a_i)\leq X\},\\
\mc P_c(X)&=\{z^d+a_{d-1}z^{d-1}+\cdots+a_0\in\mc P(X)\,|\,a_{d-1}=0\}, \\
\mc S(X)&=\PxP\setminus\{(f,f)\,|\,f\in\mc P_c(X)\}.
\end{align*}
The strategy is to show that $\left|\Prep(f)\cap\Prep(g) \right|$ is small for most pairs $(f,g) \in \mc S(X)$, and not very large for the remaining pairs. In more detail:
\begin{itemize}
\item For most pairs $(f,g)$, the coefficients of $f,g$ have a certain divisibility structure (\cref{RationalCount2}). Often this implies that the filled Julia sets of $f,g$ are disjoint at some place (\cref{main-thm-prep-case1}), so $f$ and $g$ cannot have common preperiodic points. Otherwise, we can still use quantitative equidistribution to show that $f$ and $g$ have few common preperiodic points (\cref{main-thm-prep-case2}).
\item For the remaining pairs $(f,g)$, we adapt the approach of DeMarco--Krieger--Ye \cite{DKY22} to prove
\[X^{-k-4}\lesssim_d\langle\mu_f,\mu_g\rangle\lesssim_d\frac{\log N}N\log X,\]
where $N=|\Prep(f)\cap\Prep(g)|$, and $f,g$ agree in their first $k-1$ coefficients. The upper bound comes from H\"older bounds on $G_{f,v}$ when $f$ has bad reduction at $v$ (\cref{holder-bound-all}), while the lower bound comes from the moments of $\mu_{f,\infty}$ at the archimedean place (\cref{ArchimedeanLowerBoundEnergy}). Together, these bounds yield $N\lesssim_{\eps,d}X^{k+4+\eps}$ (\cref{main-thm-prep-case3}).
\end{itemize}

\subsection{Elementary properties for most pairs}\label{sec:generic-dyn-ele}
We first formalize the notions of ``most'' and ``divisibility structure'' in the proof outline: we will require that at some place, $f$ has all coefficients with small norm, while $g$ has some coefficient with large norm.
\begin{proposition} \label{RationalCount2} 
Fix $f\in\mc P(X)$. Let $J\subseteq\{0,1,\ldots,d-1\}$ and choose $b'_i\in\bb Q$ for each $i\not\in J$ with $h(b'_i)\leq\log X$. Then for any $\eps>0$, among all polynomials $g\in\mc P(X)$ with prescribed coefficients $b_i=b'_i$ for all $i\not\in J$, at most $O_{\eps,d}(X^{-|J|+\eps})$ proportion satisfy the following property: there does not exist a non-archimedean place $v\in M_{\bb Q}$ such that
\[\forall i\,[|a_i|_v\leq1],\quad\forall i\not\in J\,[|b_i| _v\leq1],\quad\exists j\in J\,[|b_j|_v>1].\]
\end{proposition}
(Recall that $a_i,b_i$ are the coefficients of $f,g$ respectively.)
\begin{proof}
Standard results from analytic number theory (e.g., \cite[\S18.5]{hardy-wright}) yield
\begin{equation} \label{eq:asymp}
\# \{x\in \bb{Q}\,:\,h(x)\leq\log X\}=\frac{2X^2}{\zeta(2)}+O(X\log X).
\end{equation}
Hence there are $O_d(X^{2|J|})$ many $g\in\cal{P}(X)$ with prescribed coefficients $b_i=b'_i$ for $i\not \in J$, so we need to show that there are $O_{\eps,d}(X^{|J|+\eps})$ polynomials $g$ with the given property.

Let $S=\{q_1,q_2,\ldots,q_m\}$ be the set of all prime factors of the denominators of $a_i$ ($0\leq i\leq d-1$) and $b_i$ ($i\not\in J$). Then the given property is equivalent to requiring that for each $j\in J$, every prime factor of $\denom(b_j)$ is in $S$. It now suffices to show that 
there are at most $O_{\eps,d}(X^{\eps/|J|})$ integers in $[1,X]$ with all prime divisors in $S$: then for each $j\in J$, there are $O(X)$ possible numerators and $O_{\eps,d}(X^{\eps/|J|})$ possible denominators for $b_j$, so the total number of choices for $(b_j)_{j\in J}$ (and hence for $g$) is $O_{\eps,d}(X^{1+\eps/|J|})^{|J|}=O_{\eps,d}(X^{|J|+\eps})$.
 
Let $2=p_1\leq p_2\leq\cdots$ denote the primes in increasing order; then it suffices to show that there are $O_{\eps,d}(X^{\eps/|J|})$ integers in $[1,X]$ with all prime divisors in $\{p_1,\ldots,p_m\}$, i.e., which are $p_m$-smooth. But $\prod_i p_i\leq\prod_iq_i\leq X^{2d-|J|}$, so \cite[\S22.2]{hardy-wright} gives $p_m\lesssim d\log X$. Then 
by \cite[Eq.\ (1.19)]{granville}, the number of 
$p_m$-smooth integers in $[1,X]$ is
\[\lesssim 
\exp\left(C_d\frac{\log X}{\log\log X}\right)\lesssim_{\eps,d}X^{\eps/|J|},\]
and we are done.
\end{proof}

By \cref{RationalCount2}, for most pairs $(f,g)\in\PxP$, there is some place $v$ where $f$ has explicit good reduction, while $g$ has some large coefficient. 
We now prove some elementary properties of the filled Julia sets and equilibrium measures of $f$ and $g$ in the above scenario, which will be used to carry out the first step in the proof of \cref{main-thm-prep}. 

\begin{proposition} \label{NonArchimedean2}
Let $v\in M_{\bb Q}$ be non-archimedean. If $|a_i|_v\leq1$ for all $i$, then $|\zeta_v| \leq1$ for all $\zeta\in K_{f,v}$. 
\end{proposition}
\begin{proof}
By the density of Type I points in $K_{f,v}$, it suffices to prove the statement for $\zeta=z\in\bb C_v$. Now if $|z_v| >1$, then $|f(z)|_v=|z_v|^d>1$; hence by induction we have $|f^n(z)|_v=|z_v|^{d^n}\to\infty$, and thus $z\not\in K_{f,v}$. 
\end{proof}

\begin{proposition} \label{NonArchimedean1}
Let $v\in M_{\bb Q}$ be non-archimedean. If $|b_i|_v\leq1$ for all $i\neq0$ but $|b_0|_v>1$, then $|\zeta|_v=|b_0|_v^{1/d}$ for all $\zeta\in K_{g,v}$.
\end{proposition}
\begin{proof}
By the density of Type I points in $K_{g,v}$, it suffices to prove the statement for $\zeta=z\in\bb C_v$. If $|z_v| \neq|b_0|_v^{1/d}$, then
\[|g(z)|_v=\max(|z^d|_v,|b_0|_v)>|b_0|_v^{1/d};\]
hence by induction we have $|g^n(z)|_v= |g(z)|_v^{d^{n-1}}\to\infty$, and thus $z\not\in K_{g,v}$.
\end{proof}

Hence most pairs $(f,g)$ have no common preperiodic points:
\begin{corollary}\label{main-thm-prep-case1}
Fix $f\in\mc P_c(X)$. For $1-O_{\eps,d}(X^{-1+\eps})$ proportion of $g\in\mc P(X)$, there exists a non-archimedean place $v\in M_{\bb Q}$ such that
\[\forall i\,[|a_i|_v\leq1],\quad\forall j\neq0\,[|b_j|_v\leq1],\quad |b_0|_v>1.\]
Furthermore, in this case we have $\Prep(f)\cap\Prep(g)=\emptyset$.
\end{corollary}
\begin{proof}
The first claim follows by applying \cref{RationalCount2} with $J=\{0\}$. Next, by \cref{NonArchimedean2,NonArchimedean1}, we have
\begin{align*}
z\in\Prep(f)\subseteq K_{f,v}&\implies |z|_v\leq1,\\
z\in\Prep(g)\subseteq K_{g,v}&\implies |z|_v=|b_0|_v^{1/d}>1,
\end{align*}
so $\Prep(f)\cap\Prep(g)=\emptyset$.
\end{proof}

When the large coefficient of $g$ is not the constant coefficient, we see below that the distribution of mass for $\mu_{g,v}$ differs from $\mu_{f,v}=\delta_{\zeta(0,1)}$ by a definite amount. Then \cref{NonArchimedeanEnergyBound} gives an upper bound on the number of common preperiodic points. 
\begin{proposition} \label{NonArchimedean3}
Let $v\in M_{\bb Q}$ be non-archimedean, with $p$ its associated prime. If there exists $1\leq j\leq d-1$ such that $|b_j|_v>1$, then
\[\mu_{g,v}(\berkA \setminus D_{\an}(0,p^{1/d}))\geq1/d.\]
\end{proposition}
\begin{proof}
First, we claim that for any $c\in\bb C_v$, the polynomial $g-c$ has at least one root $w$ with $|w|_v\geq |b_j|_v^{1/(d-j)}\geq p^{1/d}$. This follows from considering the Newton polygon of $g-c$, which must have a side with slope at least $\frac1{d-j}\log |b_j|_v$.

Now for any non-exceptional point $z\in\bb{C}_v$, we consider the iterated pullback $d^{-n}(g^n)^{\delta_z}$. On one hand, by the above argument, this measure always assigns at least $1/d$ mass outside of $D_{\an}(0,p^{1/d})$. On the other hand, by equidistribution, this sequence of measures converges weakly to the equilibrium measure $\mu_{g,v}$; hence $\mu_{g,v}$ must also assign at least $1/d$ mass outside of $D_{\an}(0,p^{1/d})$.
\end{proof}

\begin{corollary}\label{main-thm-prep-case2}
Let $(f,g)\in\PxP$. If there exists a non-archimedean place $v\in M_{\bb Q}$ such that
\[\forall i\,[|a_i|_v\leq1],\quad\exists j>0\,[|b_j|_v>1],\]
then $|\Prep(f)\cap\Prep(g)|\lesssim_d\log^4X$.
\end{corollary}
\begin{proof}
Let $p$ be the prime associated to $v$. Then $p\mid\denom(b_j)$, so $p\leq X$.

Let $F=\Prep(f)\cap\Prep(g)$. Now $F$ is $\Gal(\Qbar/\bb Q)$-invariant, and $F\subseteq\overline B(0,1)$ by \cref{NonArchimedean1}. Hence 
by \cref{NonArchimedean3,NonArchimedeanEnergyBound}, we have
\[|F|\lesssim_d(h(g)^2+1)\log^2p\lesssim\log^4X.\qedhere\]
\end{proof}

\subsection{Proof of \cref{main-thm-prep,main-thm-prep-subspace}}

We now restate and prove \cref{main-thm-prep}.

\mainthmprep*

\begin{proof}
Fix $f\in\mc P_c(X)$.

\emph{Case 1}: For some non-archimedean place $v\in M_{\bb Q}$, we have
\[\forall i\,[|a_i|_v\leq1],\quad\forall j\neq0\,[|b_j|_v\leq1],\quad |b_0|_v>1.\]
By \cref{main-thm-prep-case1}, $f$ and $g$ have no common preperiodic points.

\emph{Case 2}: Case 1 does not hold, and for some non-archimedean place $v\in M_{\bb Q}$ we have
\[\forall i\,[|a_i|_v\leq1],\quad\exists j\,[|b_j|_v>1].\]
By \cref{main-thm-prep-case1}, this case occurs for at most $O_{\eps,d}(X^{-1+\eps})$ proportion of $g\in\mc P(X)$, and by \cref{main-thm-prep-case2}, $|\Prep(f)\cap\Prep(g)|\lesssim_d\log^4X$.

\emph{Case 3}: Neither Case 1 nor Case 2 holds. We will further split this case by the value of $k$, $1\leq k\leq d$, for which
\[\forall i<k\,[a_{d-i}=b_{d-i}],\quad a_{d-k}\neq b_{d-k}.\]
By \cref{RationalCount2} with $J=\{d-k,d-k-1,\ldots,0\}$, the subcase for each value of $k$ occurs for at most $O_{\eps,d}(X^{-2(k-1)}X^{-(d-k+1)+\eps})=O_{\eps,d}(X^{-(d+k-1)+\eps})$ proportion of $g\in\mc P(X)$, and \cref{main-thm-prep-case3} gives $|\Prep(f)\cap\Prep(g)|\lesssim_dX^{k+4+\eps}$.

Putting these three cases together, the average value of $|\Prep(f)\cap\Prep(g)|$ over all $g\in\mc P(X)\setminus\{f\}$ is at most
\[O_{\eps,d}\left(X^{-1+\eps}\log^4X+\sum_{k=1}^{d-1}X^{-(d+k-1)+\eps}X^{k+4+\eps}\right)=O_{\eps,d}(X^{-1+2\eps}+X^{5-d+\eps}).\]
For $d\geq6$ and $\eps>0$ small, this goes to 0 as $X\to\infty$, and we are done.
\end{proof}

We now prove \cref{main-thm-prep-subspace} using the same argument. Let $V$ be a linear subspace of $(a_{d-2},\ldots,a_0) \simeq \bb{Q}^{d-1}$ of dimension $m$ that is given by an intersection of coordinate hyperplanes $\{a_i = 0\}$. Assume that the projection map from $V$ to the last coordinate $a_0$ is surjective. Let $x = (x_{d-1},\ldots,x_0)$ be a point in $\bb{Q}^{d-1}$ and let $A$ be the affine subspace $x+V$. Choose $m$ coordinates $a_{i_1},\ldots,a_{i_m}$ such that $i_j$ are decreasing, $i_m = 0$, and they induce an isomorphism $V \simeq \bb{Q}^m$. We define  
$$\cal{P}_{c,A}(X) = \{z^d + a_{d-2} z^{d-2} + \cdots + a_0 \mid (0, a_{d-2},\ldots,a_0) \in A \text{ and } H(a_{i_j}) \leq X\},$$
$$\cal{P}_{A}(X) = \{z^d + a_{d-1} z^{d-1} + \cdots + a_0 \mid (a_{d-1} , \ldots, a_0) \in A \text{ and } H(a_{i_j}) \leq X\},$$
 $$\cal{S}_{V}(X) = \cal{P}_{c,A}(X) \times \cal{P}_{A}(X) \setminus \{(f,f) \mid f \in \cal{P}_{c,A}(X) \}.$$
 
\mainthmprepsubspace*

\begin{proof}
The same argument of \cref{main-thm-prep} applies. Fix $f(z) = z^d + a_{d-2} z^{d-2} + \cdots + a_0 \in \cal{P}_{c,A}(X)$ and write $g(z) = z^d + b_{d-1}z^{d-1} + \cdots  + b_0 \in \cal{P}_{A}(X)$. Then we keep Case 1 as the same as \cref{main-thm-prep}. We modify Case 2 to asking for a non-archimedean place $v$ such that $|b_{i_k}|_v > 1$ for some $k$ and $|a_i|_v \leq 1$ for all $i$. For Case 3, we split by the value of $1 \leq k \leq m$ for which
$$\forall j < k [a_{i_j} = b_{i_j}], a_{i_k} \not = b_{i_k}.$$
As $V$ is given by an intersection of coordinate hyperplanes, it follows that we can vary $b_0$ freely over rational numbers of height $\leq X$ while keeping the rest of the coefficients fixed. Thus \cref{RationalCount2} again implies Case 2 holds for $O(X^{-1+\epsilon})$ proportion of $g(z)$'s. The same bound of $O(\log^4 X)$ holds for this case as the height of each coefficient is again $O(\log X)$. Similarly as we can vary each $b_{i_k}$ freely while keeping the rest of the coefficients fixed, \cref{RationalCount2} implies that Case 3 occurs for at most 
$$O_{\epsilon,d}( X^{-2(k-1)} X^{-(m-k+1) + \epsilon}) = O_{\epsilon,d}(X^{-(m+k-1)+\epsilon})$$
proportion for all possible choices for $g$. Now by \cref{main-thm-prep-case3}, since the first index $j$ for which $a_{d-j} \not = b_{d-j}$ is at most  $d-m+k$, we get 
$$|\Prep(f) \cap \Prep(g)| \lesssim_d X^{d-m+k+4+\epsilon}.$$
Putting it together, we get an upper bound of
$$O_{\epsilon,d} \left( X^{-1+\epsilon} \log^2(X) + \sum_{k=1}^{m} X^{-(m+k-1)+\epsilon} X^{d-m+k+4+\epsilon} \right) = O_{\epsilon,d}( X^{-1 + \epsilon} + X^{5+d-2m + \epsilon}).$$
Thus if $m \geq \frac{d}{2} + 3$, this quantity goes to zero as $X \to \infty$ as desired. 
\end{proof}

\section{Generic polynomials and dynamics} \label{sec:generic}
In this section, we will define a notion of generic polynomials and study its properties, mainly its dynamics. For a pair of polynomials $(f,g)$ with rational coefficients, we first consider three notions of arithmetic complexity, namely $h(f)+h(g)$, $\langle\mu_f,\mu_g\rangle$, and $\log X$, and show that these are comparable when the pair $(f,g)$ is generic. We then make a detailed study of the equilibrium measures and potentials for generic polynomials at most places. Recall again that for a fixed degree $d$ and $X \geq 1$, we have the sets
\begin{align*}
\mc P(X)&=\{z^d+a_{d-1}z^{d-1}+\cdots+a_0\,|\,a_i\in\bb Q,\,H(a_i)\leq X\},\\
\mc P_c(X)&=\{z^d+a_{d-1}z^{d-1}+\cdots+a_0\in\mc P(X)\,|\,a_{d-1}=0\},
\end{align*}

\subsection{Complexity of generic pairs} \label{subsec:bogomolov-complexity}
The arithmetic complexity of a pair $(f,g)\in\PxP$ can be measured by the adelic energy pairing $\langle\mu_f,\mu_g\rangle$, by the heights of the polynomials $h(f)+h(g)$, or simply by the height bound $\log X$. In this section, we show that these three quantities are comparable in the generic case. We first recall the notion of $\eps$-ordinary. For an integer $n$, we let $\rad(n)$ be the product of all distinct prime factors of $n$. 

\begin{definition}\label{def:eps-ord}
Given $X$ and $\eps>0$, we will say that a pair $(f,g)\in\PxP$ is \emph{$\eps$-ordinary} if
$$
\rad(\denom(c)) \geq X^{1-2\eps} \text{  and  }
\gcd(\denom(c),\denom(c')) \leq X^{2\eps},
$$
for any distinct coefficients $c,c'$ of $f$ or $g$.
\end{definition}

\begin{theorem}\label{height-ineqs}
For any $(f,g)\in\PxP$, we have
\[\langle\mu_f,\mu_g\rangle-2\leq\frac{h(f)+h(g)}d\leq2\log X.\]
Moreover, if $(f,g)$ is $\eps$-ordinary, then
\[\langle\mu_f,\mu_g\rangle+O(d\eps\log X)\geq\frac{h(f)+h(g)}d\geq\left(2-O(d\eps+\tfrac1d)\right)\log X.\]
\end{theorem}

We prove these bounds in \cref{height-log-all,height-pairing-all,height-log-most,height-pairing-most}. For convenience, we make the following definition.

\begin{definition}
For $v\in M_K$, write $\cal{M}_{f,v}:= \max(1,|a_{d-1}|_v,\ldots,|a_0|_v)$.
\end{definition}

\begin{remark}
Note that $\log \cal{M}_{f,v}$ is the local contribution to the height of $f$ at $v$, i.e., $h(f)=\frac1{[K:\bb Q]}\sum_{v\in M_K}\!n_v\log\cal{M}_{f,v}$. Also, for any place $v'$ of a finite extension $K'/K$ with $v'\mid v$, we have $\cal{M}_{f,v'}=\cal{M}_{f,v}$.
\end{remark}

We begin with an easy bound for $h(f)$ in terms of $\log X$.
\begin{proposition}\label{height-log-all}
For any $f\in\cal{P}(X)$, we have $h(f)\leq d\log X$.
\end{proposition}
\begin{proof}
For any $v\in M_{\bb Q}$,
$$
\cal{M}_{f,v} =\max(1,|a_{d-1}|_v,\ldots,|a_0|_v) \leq\max(1,|a_{d-1}|_v)\cdots\max(1,|a_0|_v).
$$
Hence
$$
h(f)=\sum_{v\in M_{\bb Q}}\!\log \cal{M}_{f,v} \leq\sum_{v\in M_{\bb Q}}\!\left(\log^+|a_{d-1}|_v+\cdots+\log^+|a_0|_v\right) =h(a_{d-1})+\cdots+h(a_0) \leq d\log X.\qedhere
$$
\end{proof}

Now we will relate $h(f)+h(g)$ to $\langle\mu_f,\mu_g\rangle$. We first note that $\log\cal{M}_{f,v}$ can be read off from the equilibrium potential $G_{f,v}$ at the Gauss point:
\begin{lemma}\label{pairing-bound-gauss}
Let $v\in M_K$ be non-archimedean. Then
\[G_{f,v}(\zeta(0,1))=\frac1d\log\cal{M}_{f,v}.\]
\end{lemma}

\begin{proof}
It is routine to check that $f(\zeta(0,1))=\zeta(0,\cal{M}_{f,v})$, and $f(\zeta(0,R))=\zeta(0,R^d)$ for all $R\geq \cal{M}_{f,v}$. Hence for all $n\geq1$ we have
\[f^n(\zeta(0,1))=\zeta(0,\cal{M}_{f,v}^{d^{n-1}}),\]
and the result follows.
\end{proof}

On the other hand, we can bound $G_{f,v}$ in terms of its value at the Gauss point, which in turn yields a bound on the local pairing $-(\mu_f,\mu_g)_v$.
\begin{lemma}\label{green-bound-nonarch}
Let $v\in M_K$ be non-archimedean. For any $\xi\in\berkA$, we have
\[G_{f,v}(\xi)\leq\log^+|\xi|_v+G_{f,v}(\zeta(0,1)).\]
\end{lemma}
\begin{proof}
The left side of the claim is subharmonic on $\berkA$, while the right side is harmonic on $\berkA\setminus\{\zeta(0,1)\}$. Hence it suffices to check the inequality for $\xi=\zeta(0,1)$, where it is clear, and for $|\xi|_v\to\infty$. But $G_{f,v}(\xi)=\log^+|\xi|_v$ for all $\xi$ with $|\xi|_v$ sufficiently large, so we are done by observing that
\[G_{f,v}(\zeta(0,1))\geq\inf G_{f,v}=I(\mu_{f,v})=0.\qedhere\]
\end{proof}

\begin{corollary}\label{pairing-bound-nonarch}
Let $v\in M_K$ be non-archimedean. Then
\[-(\mu_f,\mu_g)_v\leq\frac1d(\log \cal{M}_{f,v}+\log \cal{M}_{g,v}).\]
\end{corollary}
\begin{proof}
We have by \cref{green-bound-nonarch,pairing-bound-gauss},
$$
-(\mu_f,\mu_g)_v =\int G_{f,v}(\xi) d\mu_{g,v}(\xi) \leq\int\log^+|\xi|_v d\mu_{g,v}(\xi)+G_{f,v}(\zeta(0,1))$$
$$=G_{g,v}(\zeta(0,1))+G_{f,v}(\zeta(0,1)) =\frac1d(\log \cal{M}_{f,v}+\log \cal{M}_{g,v}). \qedhere
$$
\end{proof}

The situation at archimedean places $v\mid\infty$ is analogous. As usual, we fix an embedding $\Qbar\hookrightarrow\bb C$, and write $|\cdot|_v=|\cdot|_\infty$.
\begin{lemma}\label{pairing-bound-max-arch}
We have
\[\max_{|w|_\infty=1}G_{f,\infty}(w)\leq\frac1d\log \cal{M}_{f,\infty}+1.\]
\end{lemma}
\begin{proof}
Write $|\cdot| =|\cdot|_\infty$ and $R=\max(2,\cal{M}_{f,\infty})$. For $|z| \geq R$, we have
$$
|f(z)| \leq |z|^d+R(|z|^{d-1}+\cdots+|z|+1) \leq(|z|+R)^d-R,
$$
by comparing terms of the binomial expansion. By induction, we have $|f^n(z)| \leq(|z|+R)^{d^n}-R$, so that by the maximum principle we have
\[G_{f,\infty}(z)\leq\log(\max(|z|,R)+R)\]
for all $z\in\bb C$.

Now for $|w|=1$, we have $|f(w)|\leq dR+1$. Hence
$$
d\cdot G_{f,\infty}(w)=G_{f,\infty}(f(w)) \leq\log(\max(|f(w)|,R)+R) \leq\log((d+1)R+1).
$$
If $1\leq \cal{M}_{f,\infty}\leq2$ then $R=2$, so
\[d\cdot G_{f,\infty}(w)\leq\log(2d+3)\leq d\leq\log \cal{M}_{f,\infty}+d.\]
Otherwise, we have $\cal{M}_{f,\infty}>2$ and $R= \cal{M}_{f,\infty}$, so
\[d\cdot G_{f,\infty}(w)\leq\log \cal{M}_{f,\infty}+\log(d+2)\leq\log \cal{M}_{f,\infty}+d.\qedhere\]
\end{proof}

\begin{lemma}\label{green-bound-arch}
For any $z\in\bb C$, we have
\[G_{f,\infty}(z)\leq\log^+|z|_\infty+\max_{|w|_\infty=1}G_{f,\infty}(w).\]
\end{lemma}
\begin{proof}
Write $|\cdot|=|\cdot|_\infty$. The left side of the claim is subharmonic on $\bb C$, while the right side is harmonic on $\bb C\setminus\del B(0,1)$. Hence it suffices to check the inequality on $|z|=1$, where it is clear, and for $|z|\to\infty$. But $G_{f,\infty}(z)-\log^+ |z|\to 0$ as $|z|\to\infty$, so we are done by observing that
\[\max_{|w|=1}G_{f,\infty}(w)\geq\inf G_{f,\infty}=I(\mu_{f,\infty})=0.\qedhere\]
\end{proof}

\begin{corollary}\label{pairing-bound-arch}
We have
\[-(\mu_f,\mu_g)_\infty\leq\frac1d(\log \cal{M}_{f,\infty}+\log \cal{M}_{g,\infty})+2.\]
\end{corollary}
\begin{proof}
We have by \cref{green-bound-arch,pairing-bound-max-arch}
$$
-(\mu_f,\mu_g)_\infty =\int G_{f,\infty}(z) d\mu_{g,\infty}(z) \leq\max_{|w|_\infty=1}G_{f,\infty}(w)+\int\log^+|z|_\infty d\mu_{g,\infty}(z)
=\max_{|w|_\infty=1}G_{f,\infty}(w)-(\mu_{S^1},\mu_{g,\infty})_\infty $$
$$=\max_{|w|_\infty=1}G_{f,\infty}(w)+\int G_{g,\infty}(e^{i\theta})\,\frac{d\theta}{2\pi}
\leq\max_{|w|_\infty=1}G_{f,\infty}(w)+\max_{|w|_\infty=1}G_{g,\infty}(w)
\leq\frac1d(\log \cal{M}_{f,\infty}+\log \cal{M}_{g,\infty})+2. \qedhere
$$
\end{proof}

Combining the above results, we obtain the following bound on the adelic energy pairing $\langle\mu_f,\mu_g\rangle$.
\begin{proposition}\label{height-pairing-all}
We have
\[\langle\mu_f,\mu_g\rangle\leq\frac{h(f)+h(g)}d+2.\]
\end{proposition}
\begin{proof}
Since $(\mu_f,\mu_f)_v=(\mu_g,\mu_g)_v=0$, we have
\[\langle\mu_f,\mu_g\rangle=\frac12\sum_{v\in M_{\bb Q}}\!(\mu_f-\mu_g,\mu_f-\mu_g)_v=\sum_{v\in M_{\bb Q}}\!-(\mu_f,\mu_g)_v.\]
Now each term in the sum is bounded by \cref{pairing-bound-nonarch,pairing-bound-arch}; summing over all places, we obtain the desired bound.
\end{proof}

Our next objective is to show that most pairs $(f,g)\in\PxP$ are $\eps$-ordinary, which requires some preparatory estimates.

\begin{proposition}[{\cite[Lem.\ 2.9]{LM21}}] \label{RadicalBound}
Fix $\eps>0$. Then for all large enough $X$,
\[\sum_{n\leq X}\frac1{\rad(n)}\leq X^\eps.\]
\end{proposition}

\begin{corollary} \label{RadicalBound2}
There are at most $O_\eps(X^{2-\eps})$ rationals $x$ such that $h(x)\leq\log X$ and $\rad(\denom(x))\leq X^{1-2\eps}$.  
\end{corollary}
\begin{proof}
By \cref{RadicalBound}, the number of possible values of $\denom(x)$ is
\[\#\{n\leq X\,:\,\rad(n)\leq X^{1-2\eps}\}\lesssim_\eps X^\eps X^{1-2\eps}=X^{1-\eps},\]
and for each such denominator there are at most $2X+1$ possibilities for $x$.
\end{proof}

\begin{lemma}\label{RationalCountLem}
Let $x_1,x_2$ be chosen uniformly at random independently from the finite set $\{x\in\bb Q\,:\,h(x)\leq\log X\}$, and write $x_i=\frac{\alpha_i}{\beta_i}$ for some $|\alpha_i| \leq X$, $1\leq \beta_i\leq X$, $\gcd(\alpha_i,\beta_i)=1$. Then
\[\Pr(\gcd(\beta_1,\beta_2)\geq X^{2\eps})\lesssim_\eps X^{-\eps}.\]
\end{lemma}

\begin{proof}
Fix $x_1$. If $\gcd(\beta_1,\beta_2)\geq X^{2\eps}$, then $\beta_2=km$ for some $k,m$ satisfying
\[m\mid\beta_1,\quad m\geq X^{2\eps},\quad km\leq X.\]
Since $\beta_1\leq X$, the number of divisors of $\beta_1$ is at most $O_\eps(X^\eps)$, by the divisor bound. Hence there are at most $O_\eps(X^\eps)$ choices for $m$, and for each such $m$ there are at most $X^{1-2\eps}$ choices for $k$. Also, for each $\beta_2$ there are at most $X$ choices for $\alpha_2$. Thus there are $\lesssim_\eps X^\eps\cdot X^{1-2\eps}\cdot X=X^{2-\eps}$ choices for $x_2$, out of $\sim X^2$ possibilities from $\{x\in\bb Q\,:\,h(x)\leq\log X\}$ by \cref{eq:asymp}. Hence
\[\Pr(\gcd(\beta_1,\beta_2)\geq X^{2\eps}\,|\,\beta_1 \text{ is fixed})\lesssim_\eps\frac{X^{2-\eps}}{X^2}=X^{-\eps},\]
which implies the desired statement.
\end{proof}

\begin{proposition}\label{ordinary-most}
The proportion of pairs $(f,g)\in\PxP$ which are $\eps$-ordinary is $1-O_\eps(d^2X^{-\eps})$.
\end{proposition}
\begin{proof}
By \cref{RadicalBound2}, there is at most an $O_\eps(dX^{-\eps})$ proportion of polynomials $f\in\mc P_c(X)$ where $\rad(\denom(a_i))\leq X^{1-2\eps}$ for some $i$; the analogous bound holds for $g$. Also, by \cref{RationalCountLem}, there is at most an $O_\eps(d^2X^{-\eps})$ proportion of pairs $(f,g)\in\PxP$ where $\gcd(\denom(a_i),\denom(b_j))>X^{2\eps}$ for some $i,j$.

Hence at most $O_\eps(d^2X^{-\eps})$ proportion of pairs $(f,g)$ satisfies one of the above conditions; this is equivalent to not being $\eps$-ordinary, and we are done.
\end{proof}

We now show that for $\eps$-ordinary pairs, the reverse inequalities to \cref{height-log-all,height-pairing-all} hold up to a small multiplicative error.

\begin{definition}\label{def:assoc}
Given a pair $(f,g)$, and $0\leq j\leq d-1$, we say that a non-archimedean place $v\in M_K$ is associated to $a_j$ if
\[\forall i\neq j\,[|a_i|_v\leq1],\quad\forall i\,[|b_i|_v\leq1],\quad |a_j|_v>1.\]
We define places associated to $b_j$ similarly.

We call a non-archimedean place $v\in M_K$ a good place for the pair $(f,g)$ if $v$ is associated to some coefficient $a_j$ or $b_j$, or a bad place for $(f,g)$ otherwise. (We will always treat archimedean places separately.)
\end{definition}

\begin{remark}
Each $v\in M_K$ is associated to at most one coefficient of $f$ or $g$. If $v$ is associated to $a_j$, then $\cal{M}_{f,v}=|a_j|_v>1$ and $\cal{M}_{g,v}=1$, and similarly for $b_j$.
\end{remark}

We first show the reverse bound to \cref{height-log-all} for $h(f)+h(g)$ and $\log X$.
\begin{lemma}\label{sum-assoc}
For any $\eps$-ordinary pair $(f,g)\in\PxP$, we have
\[\frac1{[K:\bb Q]}\sum_{\substack{v\in M_K\\v\text{ assoc.\ }a_j}}\!\!n_v\log |a_j|_v\geq(1-O(d\eps))\log X,\]
and similarly for $b_j$.
\end{lemma}

\begin{proof}
We may assume that $K=\bb Q$. Let $A_i,B_i$ denote the denominators of $a_i,b_i$ respectively. Now $v\in M_{\bb Q}$ is associated to $a_j$ if and only if $v$ divides $A_j$, but no other $A_i$ or $B_i$; hence
$$
\sum_{v\text{ assoc.\ }a_j}\!\!\log |a_j|_v \geq\log\rad(A_j)-\sum_{i\neq j}\log\gcd(A_i,A_j)-\sum_i\log\gcd(B_i,A_j)$$ 
$$\geq(1-2\eps)\log X-(2d-1)2\eps\log X \geq(1-4d\eps)\log X.
$$
The argument for $b_j$ is analogous.
\end{proof}

\begin{proposition}\label{height-log-most}
For any $\eps$-ordinary pair $(f,g)\in\PxP$, we have
\[\frac{h(f)+h(g)}d\geq\left(2-O(d\eps+\tfrac1d)\right)\log X.\]
\end{proposition}

\begin{proof}
Using \cref{sum-assoc}, we have
$$
h(f)+h(g) =\sum_{v\in M_{\bb Q}}\!(\log\mc M_{f,v}+\log\mc M_{g,v}) \geq\sum_j\sum_{v\text{ assoc.\ }a_j}\!\!\log|a_j|_v+\sum_j\sum_{v\text{ assoc.\ }b_j}\!\!\log|b_j|_v$$
$$
\geq(2d-1)(1-4d\eps)\log X \geq2d\left(1-O(d\eps+\tfrac1d)\right)\log X,
$$
and we are done after dividing by $d$.
\end{proof}

we now proceed to show the reverse bound to \cref{height-pairing-all} for $h(f)+h(g)$ and $\langle\mu_f,\mu_g\rangle$. We begin by observing an important equality case for \cref{pairing-bound-nonarch}.
\begin{proposition}\label{good-places-pairing}
Let $v\in M_K$ be non-archimedean. If $v$ is good for $(f,g)$, then
\[-(\mu_f,\mu_g)_v=\frac1d(\log\mc M_{f,v}+\log\mc M_{g,v}).\]
\end{proposition}
\begin{proof}
Assume that $g$ has explicit good reduction at $v$; in particular, $\mc M_{g,v}=1$ and $\mu_{g,v}=\delta_{\zeta(0,1)}$. Now by \cref{pairing-bound-gauss}, 
$$
-(\mu_f,\mu_g)_v=\int G_{f,v}\dd\mu_{g,v} =G_{f,v}(\zeta(0,1))
=\frac1d\log\mc M_{f,v} =\frac1d(\log\mc M_{f,v}+\log\mc M_{g,v}).
$$
The argument when $f$ has explicit good reduction is analogous.
\end{proof}

Next, for an $\eps$-ordinary pair $(f,g)$, we show that the local contribution to $h(f)+h(g)$ from bad places is small.

\begin{proposition}\label{bad-height-bound}
For any $\eps$-ordinary pair $(f,g)\in\PxP$, we have
\[\frac1{[K:\bb Q]}\sum_{\substack{v\in M_K\\v\text{ bad}}}\!n_v(\log  \cal{M}_{f,v}+\log \cal{M}_{g,v})\lesssim d^2\eps\log X.\]
\end{proposition}

\begin{proof}
We may assume that $K=\bb Q$. Write $A_i,B_i$ for the denominator of $a_i,b_i$ respectively. Let $p_v$ denote the prime associated to a place $v$. Now
\begin{equation} \label{eq:bad-height-bound}
\sum_{v\text{ bad}}\log \cal{M}_{f,v}\leq\sum_{v\text{ bad}}\log p_v+\sum_{v\text{ bad}}\max_i(0,\log |a_i|_v-\log p_v).
\end{equation}

Note that $v$ is a bad place if and only if $p_v$ divides more than one of the $A_i$ or $B_i$. Hence the first sum on the right in (\ref{eq:bad-height-bound}) is
$$
\sum_{v\text{ bad}}\log p_v \leq\sum_{i,j}\log\gcd(A_i,B_j)+\sum_{i<j}\log\gcd(A_i,A_j)+\sum_{i<j}\log\gcd(B_i,B_j) \leq4d^2\eps\log X.
$$

Also, the second term on the right in (\ref{eq:bad-height-bound}) is
$$
\sum_{v\text{ bad}}\max_i(0,\log|a_i|_v-\log p_v) \leq\sum_i\sum_{v\nmid\infty}\max(0,\log|a_i|_v-\log p_v) =\sum_i(\log A_i-\log\rad(A_i)) $$
$$\leq\sum_i(\log X-\log X^{1-2\eps}) 
\leq2d\eps\log X.
$$

Substituting back into (\ref{eq:bad-height-bound}), we get $\sum_{v\text{ bad}}\log\cal{M}_{f,v}\lesssim d^2\eps\log X$. The same argument works for $g$.
\end{proof}

\begin{proposition}\label{height-pairing-most}
For any $\eps$-ordinary pair $(f,g)\in\PxP$, we have
\[\langle\mu_f,\mu_g\rangle\geq\frac{h(f)+h(g)}d-O((d\eps+\tfrac1d)\log X).\]
\end{proposition}
\begin{proof}
We may assume that $K=\bb Q$. Note that $(\mu_f,\mu_f)_v=(\mu_g,\mu_g)_v=0$ for any $v\in M_{\bb Q}$, so $0\leq(\mu_f-\mu_g,\mu_f-\mu_g)_v=-2(\mu_f,\mu_g)_v$. Hence by \cref{good-places-pairing}
$$
\langle\mu_f,\mu_g\rangle  \geq\sum_{v\text{ good}}-(\mu_f,\mu_g)_v =\frac1d\sum_{v\text{ good}}(\log\mc M_{f,v}+\log\mc M_{g,v}) =\frac{h(f)+h(g)}d-\frac1d\sum_{v\text{ bad}}(\log\mc M_{f,v}+\log\mc M_{g,v}).
$$
Using \cref{bad-height-bound}, we get the above expression is
$$
=\frac{h(f)+h(g)}{d}-\frac1d(\log\mc M_{f,\infty}+\log\mc M_{g,\infty}) \geq\frac{h(f)+h(g)}d-O(d\eps\log X)-O(\tfrac1d\log X). \qedhere
$$

\end{proof}

\subsection{Dynamics on generic polynomials} \label{subsec:generic-dyn-pot}
For $\eps$-ordinary pairs $(f,g)$, since the denominators of the coefficients are almost pairwise coprime, most primes which divide one such denominator would not divide any other denominator. Hence for most places $v$ where $f,g$ do not both have explicit good reduction, it will be the case that one of them has explicit good reduction, while the other has exactly one large coefficient. 

We know in the former that the equilibrium measure and potential are $\delta_{\zeta(0,1)}$ and $\log^+|\cdot|_v$ respectively; in this section, we will study the equilibrium measure and potential for the latter case.

First, we show that the equilibrium measure lies in three distinct strata (cf. \cref{NonArchimedean1}, which corresponds to the case $j=0$).
\begin{proposition} \label{NonArchimedean1b}
Let $v\in M_{\bb Q}$ be non-archimedean with $|a_0|_v=1$. If there exists $1\leq j\leq d-1$ such that
\[\forall i\neq j\,[|a_i|_v\leq1],\quad |a_j|_v>1,\]
then $|\zeta|_v\in\{|a_j|_v^{\frac1{d-j}},|a_j|_v^{\frac1j(\frac1{d-j}-1)},|a_j|_v^{-\frac1j}\}$ for all $\zeta\in K_{f,v}$.
\end{proposition}

\begin{proof}
By the density of Type I points in $K_{f,v}$, it suffices to prove the statement for $\zeta=z\in\bb C_v$. We will prove the contrapositive statement, so suppose that $|z|_v$ is not one of the listed values.

\emph{Case 1a}: $|z|_v\geq1$ and $j<d-1$. Then $|z^d|_v\not = |a_jz^j|_v$ by assumption, so
\[|f(z)|_v=\max(|z^d|_v,|a_jz^j|_v)>|a_j|_v^{\frac1{d-j}}.\]
By induction, we have $|f^n(z)|_v=|f(z)|_v^{d^{n-1}}$; in particular, $f^n(z)\to\infty$, so $z\not\in K_{f,v}$.

\emph{Case 1b}: $|z|_v\geq1$ and $j=d-1$. Note that $|z|_v\not =  |a_j|_v^{\frac1j(\frac1{d-j}-1)}=1$, so $|z|_v>1$, and now the argument of Case 1a works verbatim.

\emph{Case 2}: $|z|_v<1$. Then $|a_jz^j| _v \not = 1$ and $|a_0|_v=1$ by assumption, so
\[|f(z)|_v=\max(|a_jz^j|_v,1)\geq1.\]
Furthermore, $|a_jz^j|_v \not = |a_j|_v^{\frac1{d-j}}$ by assumption, so $f(z)\not\in K_{f,v}$ by Cases 1a and 1b. Thus $z\not\in K_{f,v}$ as well.
\end{proof}

The next result shows that if some subset of $\berkA$ decomposes into finitely many strata, then the equilibrium measure restricted to each stratum is a multiple of the equilibrium measure on that stratum.

\begin{lemma}\label{strata-equilibrium}
Let $r_1>r_2>\cdots>r_k>0$, and let $K_i\subseteq \overline{D}_{\an}(0,r_i)\setminus D_{\an}(0,r_i)$ be compact sets in $\berkA$ with equilibrium measures $\mu_i$. Then the equilibrium measure $\mu$ on $K=\bigcup_{i=1}^kK_i$ is a convex combination of the $\mu_i$.
\end{lemma}

\begin{proof}
Write $\mu=\sum_{i=1}^k\alpha_i\nu_i$ for some $\alpha_i\in[0,1]$ and $\nu_i$ probability measures supported on $\overline{D}_{\an}(0,r_i)\setminus D_{\an}(0,r_i)$. Now for all $i<j$ and all $\zeta_i\in K_i$, $\zeta_j\in K_j$, we have $\delta(\zeta_i,\zeta_j)_\infty=\log r_i$, so
\begin{equation} \label{eq:strata-equilibrium}
I(\mu)=\sum_{i=1}^k\alpha_i^2I(\nu_i)+\sum_{\mathclap{1\leq i<j\leq k}}2\alpha_i\alpha_j\log r_i.
\end{equation}

Since $\mu$ is the equilibrium measure on $K$, the above expression is minimal over all choices of $\nu_i$ supported on $K_i$. But the minimizer of $I(\nu_i)$ for $\nu_i$ supported on $K_i$ is precisely $\mu_i$; hence $\nu_i=\mu_i$ when $\alpha_i\neq0$, and so $\mu=\sum_{i=1}^k\alpha_i\mu_i$.

\end{proof}

\begin{remark}
The above result is known in more general contexts, cf.\ \cites[Prop.\ 4.1.27]{rumely1989}[Prop.\ A.12]{rumely2013}.
\end{remark}

\begin{theorem}\label{strata-poly}
Let $f,v,j$ be as in \cref{NonArchimedean1b}, and write
\[r_1=|a_j|_v^{\frac1{d-j}},\qquad r_2=|a_j|_v^{\frac1j(\frac1{d-j}-1)},\qquad r_3=|a_j|_v^{-\frac1j}.\]
Let $\mu_{f,v}=\alpha_1\mu_1+\alpha_2\mu_2+\alpha_3\mu_3$ for some $\alpha_i\in[0,1]$ and $\mu_i$ probability measures supported on the strata $S_i := \ovl{D}_{\an}(0,r_i)\setminus D_{\an}(0,r_i)$. Then
\begin{align*}
\alpha_1&=\frac{d-j}d,&I(\mu_1)&=-\frac j{(d-j)^2}\log|a_j|_v,\\
\alpha_2&=\frac{j(d-j)}{d^2},&I(\mu_2)&=-\left(\frac1j+\frac1{(d-j)^2}\right)\log|a_j|_v,\\
\alpha_3&=\frac{j^2}{d^2},&I(\mu_3)&=-\left(\frac1j+\frac1{j^2}\right)\log |a_j|_v.
\end{align*}
\end{theorem}

\begin{proof}
For $c\in K_{f,v}\cap\bb C_v$, we consider the Newton polygon for the polynomial $f-c$. This is the same as the Newton polygon for $f$ except when $|c|_v=r_1>1$, where the constant coefficient then has absolute value $r_1$. The two possibilities are shown in \cref{fig:newton}.

\begin{figure}[h]
\centering
\begin{tikzpicture}[yscale=1.2]
\draw[->] (0,1)--(0,-1.8) node[below]{$\log|\cdot|_v$};
\draw[->] (-0.1,0) node[left]{$0$} --(6,0);
\fill (0,0) circle (0.05)
    (0,-0.5) circle (0.05)
    (3,-1.2) circle (0.05)
    (5,0) circle (0.05);
\draw (3,0.1)--(3,-0.1) node[below]{$j$}
    (5,0.1)--(5,-0.1) node[below]{$d$}
    (0.1,-0.5)--(-0.1,-0.5) node[left]{$\frac1{d-j}\log|a_j|_v$}
    (0.1,-1.2)--(-0.1,-1.2) node[left]{$\log |a_j|_v$};
\draw[very thick] (0,1)--(0,-0.5)--(3,-1.2)--(5,0)--(5,1) (0,0)--(3,-1.2);
\end{tikzpicture}
\caption[Newton polygon for $f-c$ for $c\in K_{f,v}\cap\bb C_v$.]{Newton polygon for $f-c$ for $c\in K_{f,v}\cap\bb C_v$, when $\lvert c\rvert_v=r_1$ (below), or $\lvert c\rvert_v=r_2$ or $r_3$ (above).}
\label{fig:newton}
\end{figure}

We conclude that:
\begin{itemize}
    \item If $|c|_v=r_1$, then $f-c$ has $d-j$ roots with absolute value $r_1$, and $j$ roots with absolute value $r_2$;
    \item If $|c|_v=r_2$ or $r_3$, then $f-c$ has $d-j$ roots with absolute value $r_1$, and $j$ roots with absolute value $r_3$.
\end{itemize}

Hence for any measure $\rho$ with finite support in the Type I points within the three strata $S_1\cup S_2\cup S_3$, we have
\[\begin{pmatrix}(f^*\rho)(S_1)\\(f^*\rho)(S_2)\\(f^*\rho)(S_3)\end{pmatrix}=\begin{pmatrix}d-j&d-j&d-j\\j&0&0\\0&j&j\end{pmatrix}\begin{pmatrix}\rho(S_1)\\\rho(S_2)\\\rho(S_3)\end{pmatrix}.\]

Now let $\rho=d^{-n}(f^n)^*\delta_z$ for any non-exceptional $z\in\bb C_v$ with (say) $|z|_v=r_1$. Taking $n\to\infty$ and using equidistribution, we see that the above identity holds for $\rho=\mu_{f,v}$ as well. But $f^*\mu_{f,v}=d\cdot\mu_{f,v}$, so we have
\[d\begin{pmatrix}\alpha_1\\\alpha_2\\\alpha_3\end{pmatrix}=\begin{pmatrix}d-j&d-j&d-j\\j&0&0\\0&j&j\end{pmatrix}\begin{pmatrix}\alpha_1\\\alpha_2\\\alpha_3\end{pmatrix}.\]
Combined with $\alpha_1+\alpha_2+\alpha_3=1$, we find that indeed
\[\begin{pmatrix}\alpha_1\\\alpha_2\\\alpha_3\end{pmatrix}=\frac1{d^2}\begin{pmatrix}d(d-j)\\j(d-j)\\j^2\end{pmatrix}.\]

We now compute the energies $I(\mu_i)$. By \cref{eq:strata-equilibrium},
$$
0=I(\mu_{f,v})=I(\alpha_1\mu_1+\alpha_2\mu_2+\alpha_3\mu_3) =\alpha_1^2I(\mu_1)+\alpha_2^2I(\mu_2)+\alpha_3^2I(\mu_3) $$
$$+2\alpha_1\alpha_2\log(r_1)+2\alpha_1\alpha_3\log(r_1)+2\alpha_2\alpha_3\log(r_2).
$$
But taking the sum of $\alpha_i$ times the $i$-th expression above, we recover the expression for $I(\mu_{f,v})$, which has value $0$; hence each of the above expressions is $0$. Now we get the claimed values of $I(\mu_i)$ by isolating them within each expression.

\end{proof}

\begin{corollary}\label{non-arch-green}
In the setting of \cref{strata-poly}, we have
\[G_{f,v}(\xi)\leq\begin{cases}
\log |\xi|_v&\text{if } |\xi|_v\in[r_1,\infty)\\
\frac jd\log |\xi|_v+\frac1d\log |a_j|_v&\text{if }|\xi|_v\in[r_2,r_1]\\
\frac{j^2}{d^2}\log|\xi|_v+\frac{j+1}{d^2}\log|a_j|_v&\text{if }|\xi|_v\in[r_3,r_2]\\
\frac1{d^2}\log|a_j|_v&\text{if }|\xi|_v\in[0,r_3],
\end{cases}\]
with equality when $|\xi|_v\not\in\{r_1,r_2,r_3\}$.
\end{corollary}

\begin{proof}
We have $G_{f,v}=\alpha_1G_{\mu_1}+\alpha_2G_{\mu_2}+\alpha_3G_{\mu_3}$, where $G_{\mu_i}$ is the equilibrium potential for $\mu_i$. Since $|w|_v=r_i$ for all $w\in\supp\mu_i$, we have 
\[G_{\mu_i}(\xi)=\int\delta(\xi,w)_\infty d\mu_i(w)\leq\log\max(|\xi|_v,r_i),\]
with equality when $|\xi|_v\neq r_i$. The claim follows by routine computation.
\end{proof}

We conclude this section by computing the capacities and potentials on some sets we will use in \cref{sec:bogomolov}.

\begin{lemma}\label{strata-equilibrium-2}
Let $s_1>s_2>0$, and let $K_i\subseteq \ovl{D}_{\an}(0,s_i)\setminus D_{\an}(0,s_i)$ be compact sets with equilibrium measures $\nu_i$. Then the equilibrium measure $\nu$ on $K=K_1\cup K_2$ is given by $\nu=\alpha_1\nu_1+\alpha_2\nu_2$ for some $\alpha_1,\alpha_2\in[0,1]$. 

Moreover, we have
\[I(\nu)=\frac{\log^2s_1-I(\nu_1)I(\nu_2)}{2\log s_1-I(\nu_1)-I(\nu_2)}.\]
\end{lemma}

\begin{proof}
By \cref{strata-equilibrium}, $\nu$ has the given form. Now
\begin{equation} \label{eq:strata-equilibrium-2}
I(\nu)=a\alpha_1^2+b\alpha_2^2+2c\alpha_1\alpha_2,\
\end{equation}
where $a=I(\nu_1)$, $b=I(\nu_2)$, and $c=\log s_1$. It is easy to verify (e.g., by Lagrange multipliers) that this quadratic form is minimized on $\alpha_1+\alpha_2=1$ at
\[(\alpha_1,\alpha_2)=\left(\frac{c-b}{2c-a-b},\frac{c-a}{2c-a-b}\right).\]
Since $K_1,K_2\subseteq \ovl{D}_{\an}(0,s_1)$, we have $a,b\leq V(\ovl{D}_{\an}(0,s_1))=c$, so $\alpha_1,\alpha_2\in[0,1]$. Finally, substituting the values of $\alpha_1,\alpha_2$ into \cref{eq:strata-equilibrium-2} yields $I(\nu)=\frac{c^2-ab}{2c-a-b}$.
\end{proof}

\begin{proposition} \label{NonArchimedean4}
In the setting of \cref{strata-poly}, define $\mc R= |a_j|_v^{j/(d^2-j^2)}$,
$$
S_{f,v}^\cup =\{\zeta(0,\mc R)\}\cup\supp\mu_1, \qquad S_{f,v}^\cap =\supp\mu_2\cup\supp\mu_3.
$$
Then
\begin{align*}
\sup_{S_{f,v}^\cup}(G_{f,v}+\log^+|\cdot|_v)&=\frac1{d-j}\log|a_j|_v,&V(S_{f,v}^\cup)&=\frac1{2(d-j)}\log|a_j|_v,\\
\sup_{\mathclap{\ovl{D}_{\an}(0,1)}}(G_{f,v}+\log^+|\cdot|_v)&=\frac1d\log|a_j|_v,&V(\ovl{D}_{\an}(0,1))&=0,\\
\sup_{S_{f,v}^\cap}(G_{f,v}+\log^+|\cdot|_v)&=0,&V(S_{f,v}^\cap)&=-\frac1j\log|a_j|_v.
\end{align*}
\end{proposition}

\begin{proof}
Write $\nu_{f,v}^\cup,\nu_{f,v}^\cap$ for the equilibrium measures on $S_{f,v}^\cup,S_{f,v}^\cap$ respectively. 
Then
\[\nu_{f,v}^\cup=\alpha\mu_1+\beta\delta_{\zeta(0,\mc R)},\qquad\nu_{f,v}^\cap=\alpha'\mu_2+\beta'\mu_3.\]

By \cref{non-arch-green}, the values of $G_{f,v}$ and $\log^+|\cdot|_v$ on each of the pieces are given as follows. These yield the identities in the left column of the claim.

\[\begin{array}{rcc}
&G_{f,v}&\log^+|\cdot|_v\\\hline
\zeta(0,\mc R)&\frac d{d^2-j^2}\log |a_j|_v&\frac j{d^2-j^2}\log |a_j|_v\\
\supp\mu_1&0&\frac1{d-j}\log|a_j|_v\\
\ovl{D}_{\an}(0,1)&\leq\frac1d\log|a_j|_v&0\\
\supp\mu_2&0&0\\
\supp\mu_3&0&0
\end{array}\]

Next, we have $V(\ovl{D}_{\an}(0,1))=\log1=0$. Also, by \cref{strata-equilibrium-2},

$$
I(\nu_{f,v}^\cup) =\frac{\log^2r_1-I(\mu_1)\log\mc R}{2\log r_1-I(\mu_1)-\log\mc R},\quad I(\nu_{f,v}^\cap) =\frac{\log^2r_2-I(\mu_2)I(\mu_3)}{2\log r_2-I(\mu_2)-I(\mu_3)}.
$$
By \cref{strata-poly}, these expressions simplify to the right column of the claim.
\end{proof}

\section{Shared Small Points on average} \label{sec:bogomolov}
Let $f,g$ be polynomials of degree $d \geq 2$ over $\bb{Q}$. We will now prove \cref{main-thm-bogomolov} on the essential minimum of $\widehat h_f + \widehat h_g$ for generic pairs $(f,g)$. We first start off with some results that apply to all pairs $(f,g)$, before specializing to the generic case. 
\par 
Let $\ovl{L}_{f}$ be the canonical adelic line bundle associated to our polynomial $f \in \bb{Q}[z]$ of degree $d \geq 2$. For background on adelic line bundles, readers may refer to \cite{PST12}. Then the height associated to $\ovl{L}_{f}$ is the canonical height $h_f$, and so applying Zhang's inequality \cite[Thm.\ 1.10]{zhang95} to the line bundle $\ovl{L}_{f \otimes g} = \ovl{L}_f \otimes \ovl{L}_g$ gives us
\[\liminf_{\bb P^1(\ovl{\bb{Q}})}\left(\widehat h_f+\widehat h_g\right)+\inf_{\bb P^1(\ovl{\bb{Q}})}\left(\widehat h_f+\widehat h_g\right)\leq\langle\mu_f,\mu_g\rangle\leq2\liminf_{\bb P^1(\ovl{\bb{Q}})}\left(\widehat h_f+\widehat h_g\right).\]
Since $f$ and $g$ are both polynomials, $\infty$ is a common preperiodic point and thus we have $\inf (\widehat f + \widehat g) = 0$. We then get
\begin{equation} \label{eq:ZhangIneq1}
\frac{1}{2} \langle \mu_f, \mu_g \rangle \leq \liminf_{\bb{P}^1(\ovl{\bb{Q}})}  \left( \widehat h_f + \widehat h_g \right) \leq \langle \mu_f, \mu_g \rangle.
\end{equation}

When $\langle \mu_f ,\mu_g \rangle = 0$, this implies that $\liminf  (\widehat h_f + \widehat h_g ) = 0$, but in general only gives a possible range. We now give a different proof of (\ref{eq:ZhangIneq1}) using equidistribubtion and the energy pairing. We first prove the upper bound, which is an easy consequence of the Arakelov--Zhang pairing in \cite{PST12}.

\begin{proposition}\label{bogomolov-upper-bound-all}
$\displaystyle\liminf_{x\in\Qbar}\left(\widehat h_f(x)+\widehat h_g(x)\right)\leq\langle\mu_f,\mu_g\rangle$.
\end{proposition}
\begin{proof}
Let $(x_n)$ be a sequence of pairwise distinct points in $\Qbar$ with $\widehat h_f(x_n)\to0$ (for instance from $\Prep(f)$, which is infinite). 
By \cite[Thm.\ 1]{PST12}, we have $\widehat h_g(x_n)\to\langle\mu_f,\mu_g\rangle$, so $\langle\mu_f,\mu_g\rangle$ is a limit point for the values of $\widehat h_f+\widehat h_g$ on $\Qbar$.
\end{proof}

We next do the lower bound.

\begin{proposition}\label{bogomolov-lower-bound-all}
$\displaystyle\liminf_{x\in\bb{\ovl{Q}}}\left(\widehat h_f(x)+\widehat h_g(x)\right)\geq\frac12\langle\mu_f,\mu_g\rangle$.
\end{proposition}
\begin{proof}
Given $x\in \bb{\ovl{Q}}$, write $F\subseteq \bb{\ovl{Q}}$ for its $\Gal(\Qbar/\bb Q)$-orbit, and let $N=|F|$. Let $K$ be a number field containing $F$, and let $\eps=\{\eps_v\}_{v\in M_K}$ be the adelic radius given in the proof of \cref{UpperBoundEnergyPairing}.

Recall the parallelogram law for bilinear forms $(\cdot,\cdot)$:
\[(a+b,a+b)+(a-b,a-b)=2(a,a)+2(b,b).\]
Applying this for $a=\frac{\mu_f+\mu_g}2-[F]_\eps$, $b=\frac{\mu_f-\mu_g}2$, $(\cdot,\cdot)=(\cdot,\cdot)_v$, we have
\begin{equation}\label{eq:parallelogram}
(\mu_f-[F]_\eps,\mu_f-[F]_\eps)_v+(\mu_g-[F]_\eps,\mu_g-[F]_\eps)_v \end{equation}
$$=2\left(\frac{\mu_f+\mu_g}2-[F]_\eps,\frac{\mu_f+\mu_g}2-[F]_\eps\right)_v+\frac12(\mu_f-\mu_g,\mu_f-\mu_g)_v \geq\frac12(\mu_f-\mu_g,\mu_f-\mu_g)_v.
$$
Summing over $v\in M_K$ and using \cref{quant-equid-bound}, we obtain
$$
\frac12\langle\mu_f,\mu_g\rangle \leq\langle\mu_f,[F]_\eps\rangle+\langle\mu_g,[F]_\eps\rangle \leq\widehat h_f(x)+\widehat h_g(x)+O\left(d\frac{\log N}N(h(f)+h(g)+1)\right) \leq\widehat h_f(x)+\widehat h_g(x)+C_{d,f,g}\frac{\log N}N.
$$
Now for any $\delta>0$, choose $M$ sufficiently large such that $\frac{\log m}m\leq\frac\delta {C_{d,f,g}}$ for all $m\geq M$. Then we have the implication
\[\widehat h_f(x)+\widehat h_g(x)\leq\frac12\langle\mu_f,\mu_g\rangle-\delta\implies N=[\bb Q(x):\bb Q]<M.\]
By Northcott's theorem, there are only finitely many $x\in\Qbar$ satisfying both the hypothesis and our degree bound. Hence
\[\liminf_{x\in\Qbar}\left(\widehat h_f(x)+\widehat h_g(x)\right)\geq\frac12\langle\mu_f,\mu_g\rangle-\delta,\]
and we are done since $\delta>0$ is arbitrary.
\end{proof}

The proof of \cref{bogomolov-lower-bound-all} also explains why it is not an equality in general. For (\ref{eq:parallelogram}) to be an equality, we must have $[F] \approx \frac{1}{2}(\mu_f + \mu_g)$ at every place $v$, but then $ \sum_{v \in M_K} N_v([F],{F})_v \approx \sum_{v \in M_K} N_v \langle \mu_f , \mu_g \rangle_v$ which is positive whenever $\mu_f \not = \mu_g$, contradicting the product formula. 
\par 
We now move onto the case of a generic pair. We start by controlling the archimedean place. Recall from \cref{JuliaSetBound1Arch} that $K_{f,\infty}\subseteq B(0,R_{f,\infty})$, where $R_{f,\infty}=3\max(1,|a_{d-1}|_\infty,|a_{d-2}|_\infty^{1/2},\ldots,|a_0|_\infty^{1/d})$.

\begin{lemma}\label{rf-arch-bound-most}
For any $\eps$-ordinary pair $(f,g)\in\PxP$, we have $R_{f,\infty}\leq3X^\eps$ and $R_{g,\infty}\leq3X^{2\eps}$.
\end{lemma}
\begin{proof}
Since $f\in\mc P(X)$, the numerator of $a_j$ has absolute value at most $X$. Also, since $(f,g)$ is $\eps$-ordinary, we have $\denom(a_j)\geq X^{1-2\eps}$. Hence $|a_j|_\infty\leq X^{2\eps}$, and now $a_{d-1}=0$ implies $R_{f,\infty}\leq3X^\eps$. The argument for $g$ is analogous.
\end{proof}

The key step in the lower bound for $\liminf_\Qbar(\widehat h_f+\widehat h_g)$ is the following estimate on the sum of local heights, with an extra term $-\tfrac12\log |x|_v$ which vanishes by the product formula when we sum over all places.

\begin{lemma}\label{lem-fudge}
For any $\eps$-ordinary pair $(f,g)\in\PxP$, and any $v\in M_K$ associated to $a_j$, we have
\[\inf_{x\in\berkA}\left(G_{f,v}(x)+G_{g,v}(x)-\tfrac12\log |x|_v\right)=\begin{cases}
\frac1{2(d-j)}\log |a_j|_v&j\leq\frac{d-1}2,\\\frac{d-j-1}{2j(d-j)}\log|a_j|_v&j\geq\frac{d-1}2.\end{cases}\]
The same statement holds with $a_j$ replaced by $b_j$.
\end{lemma}

\begin{proof}
Let $\psi_v(\zeta)=G_{f,v}(\zeta)+G_{g,v}(\zeta)-\frac12\log|\zeta|_v$. Since $\psi_v(\zeta)\to\infty$ as $|\zeta|_v\to\infty$, the minimum value of $\psi_v$ occurs on the support of the positive part of $\Lap\psi_v$, which is $J_{f,v}\cup\{\zeta(0,1)\}$. 
Now $v$ is associated to $a_j$, so $g$ has explicit good reduction at $v$, and thus $G_{g,v}=\log^+|\cdot|_v$. Hence we can compute the value of $\psi_v$ at each of the three strata of $J_{f,v}$ (\cref{strata-poly}) and at $\zeta(0,1)$ (\cref{non-arch-green}), as follows:
\[\begin{array}{*5c}
|\zeta|_v&G_{f,v}(\zeta)&G_{g,v}(\zeta)&\log |\zeta|_v&\psi_v(\zeta)\\\hline
r_1&0&\frac1{d-j}\log |a_j|_v&\frac1{d-j}\log |a_j|_v&\frac1{2(d-j)}\log|a_j|_v\\
1&\frac1d\log|a_j|_v&0&0&\frac1d\log|a_j|_v\\
r_2&0&0&\frac1j(\frac1{d-j}-1)\log|a_j|_v&\frac1{2j}(1-\frac1{d-j})\log|a_j|_v\\
r_3&0&0&-\frac1j\log|a_j|_v&\frac1{2j}\log|a_j|_v
\end{array}\]

Since $\frac2d\geq\min(\frac1j,\frac1{d-j})$, and $\frac1j>\frac1j(1-\frac1{d-j})$, we have
\[\inf_\berkA\psi_v=\min\left(\tfrac1{2j}\left(1-\tfrac1{d-j}\right),\tfrac1{2(d-j)}\right)\log|a_j|_v,\]
which is equivalent to the claim. The argument for $b_j$ is analogous.
\end{proof}

\begin{theorem}\label{lower-bound-most}
For any $\eps$-ordinary pair $(f,g)\in\PxP$, we have
\[\inf_{x\in\Qbar\setminus\{0\}}\left(\widehat h_f(x)+\widehat h_g(x)\right)\geq2\left(\log2-O(d^2\eps+\tfrac{\log d}d)\right)\log X-O(1).\]
\end{theorem}

\begin{proof}
Let $K$ be a number field with $x\in K\setminus\{0\}$. For any $v\in M_K$, write $\psi_v(x)=G_{f,v}(x)+G_{g,v}(x)-\frac12\log |x|_v$. Then by the product formula $\sum_{v\in M_K}n_v\log |x|_v=0$, we have
\[\widehat h_f(x)+\widehat h_g(x)=\frac1{[K:\bb Q]}\sum_{v\in M_K}\!n_v\psi_v(x),\]
We split the sum over good, bad, and archimedean places. For good places, by \cref{lem-fudge} we have
$$\frac1{[K:\bb Q]}\sum_{v\in M_K\text{ good}}\!\!n_v\psi_v(x) \geq \frac1{[K:\bb Q]}\left(\sum_{\substack{v\text{ assoc.\ }a_j\\0<j\leq\frac{d-1}2}}\!\!n_v\frac1{2(d-j)}\log|a_j|_v+\!\!\sum_{\substack{v\text{ assoc.\ }b_j\\0<j\leq\frac{d-1}2}}\!\!n_v\frac1{2(d-j)}\log|b_j|_v \right.$$
$$+\left. \!\!\sum_{\substack{v\text{ assoc.\ }a_j\\\frac{d-1}2<j<d}}\!\!n_v\frac{d-j-1}{2j(d-j)}\log|a_j|_v+\!\!\sum_{\substack{v\text{ assoc.\ }b_j\\\frac{d-1}2<j<d}}\!\!n_v\frac{d-j-1}{2j(d-j)}\log|b_j|_v\right).
$$
By \cref{sum-assoc}, this is 
\begin{equation} \label{eq:lower-bound-most-1a}
\geq 2(1-O(d\eps))\left(\sum_{0<j\leq\frac{d-1}2}\frac1{2(d-j)}+\sum_{\frac{d-1}2<j<d}\frac{d-j-1}{2j(d-j)}\right)\log X.
\end{equation}
We recognize the first sum in (\ref{eq:lower-bound-most-1a}) as a Riemann sum:
\begin{equation} \label{eq:lower-bound-most1b}
\sum_{0<j\leq\frac{d-1}2}\frac1{2(d-j)}\geq\int_0^{1/2}\frac1{2(1-t)} dt-O\left(\frac1d\right)=\frac{\log2}2-O\left(\frac1d\right).
\end{equation}
The summand in the second sum of (\ref{eq:lower-bound-most-1a}) splits into $\frac1{2j}-\frac1{2j(d-j)}$. As before,

\begin{equation} \label{eq:lower-bound-most1c}
\sum_{\frac{d-1}2<j<d}\frac1{2j}\geq\int_{1/2}^1\frac1{2t} dt-O\left(\frac1d\right)=\frac{\log2}2-O\left(\frac1d\right),
\end{equation}
while
\begin{equation} \label{eq:lower-bound-most1d}
\sum_{\frac{d-1}2<j<d}\frac1{2j(d-j)}\leq\sum_{\frac{d-1}2<j<d}\frac1{j(d-j)}\leq O\left(\frac{\log d}d\right).
\end{equation}

Substituting (\ref{eq:lower-bound-most1b}), (\ref{eq:lower-bound-most1c}), (\ref{eq:lower-bound-most1d}) back into (\ref{eq:lower-bound-most-1a}), we have
\begin{equation} \label{eq:lower-bound-most1}
\frac1{[K:\bb Q]}\sum_{v\in M_K\text{ good}}\!\!n_v\psi_v(x)\geq2(1-O(d\eps))\left(\log2-O(\tfrac{\log d}d)\right)\log X 
\end{equation}
$$
 \geq2\left(\log2-O(d\eps+\tfrac{\log d}d)\right)\log X.
$$

For bad places, we have $K_{f,v}\subseteq \ovl{D}_{\an}(0,R_{f,v})\subseteq \ovl{D}_{\an}(0,\mc M_{f,v})$ by \cref{JuliaSetBound1NonArch}, and similarly for $g$. Hence $G_{f,v}(x)=G_{g,v}(x)=\log^+|x|_v$ for all $x$ such that $|x|_v>\max(\mc M_{f,v},\mc M_{g,v})$. Therefore
\[\psi_v(x)\geq\begin{cases}\tfrac32\log|x|_v&\text{if }|x|_v>\max(\mc M_{f,v},\mc M_{g,v})\\-\frac12\log\max(\mc M_{f,v},\mc M_{g,v})&\text{if }|x|_v\leq\max(\mc M_{f,v},\mc M_{g,v}).\end{cases}\]
Hence using \cref{bad-height-bound},
\begin{equation} \label{eq:lower-bound-most2}
\frac1{[K:\bb Q]}\sum_{v\in M_K\text{ bad}}\!\!n_v\psi_v(x) \geq-\frac1{2[K:\bb Q]}\sum_{v\in M_K\text{ bad}}\!\!n_v\left(\log\mc M_{f,v}+\log\mc M_{g,v}\right) \geq-O(d^2\eps)\log X.
\end{equation}

For archimedean places $v\mid\infty$ we split into two cases. If $|x|_v\leq R_{f,\infty}+1$ then
\[\psi_v(x)\geq-\tfrac12\log|x|_v\geq-\tfrac12\log(R_{f,\infty}+1).\]

If $|x|_v\geq R_{f,\infty}+1$, then for all $w\in K_{f,v}$ we have $|w/x|_v\leq R_{f,\infty}/(R_{f,\infty}+1)$, so
\[G_{f,v}(x)-\log|x|_v\geq\int\log|\frac{x-w}x|_v d\mu_{f,v}(w)\geq-\log(R_{f,\infty}+1),\]
and thus
\[\psi_v(x)\geq \frac12(G_{f,v}(x)-\log|x|_v) \geq-\frac12\log(R_{f,\infty}+1).\]
In either case, by \cref{rf-arch-bound-most}, we have
\begin{equation} \label{eq:lower-bound-most3}
\frac1{[K:\bb Q]}\sum_{v\mid\infty}n_v\psi_v(x) \geq-\frac12\log(R_{f,\infty}+1) =-O(\eps\log X+1).
\end{equation}

Adding (\ref{eq:lower-bound-most1}), (\ref{eq:lower-bound-most2}) and (\ref{eq:lower-bound-most3}) yields the desired bound.
\end{proof}

We now prove our upper bound of $\liminf \widehat h_f(x) + \widehat h_g(x)$, which is improves \cref{bogomolov-upper-bound-all} by a constant. As in the proof of \cref{bogomolov-upper-bound-all}, we need to construct infinitely many $x\in\Qbar$ such that $\widehat h_f(x)+\widehat h_g(x)$ is small. Our main tool will be \cref{adelic-fekete-szego}, which is a special case of the adelic Fekete--Szeg\H{o} theorem \cite[Thm.\ 0.4]{rumely2013}. 

\begin{definition}
A compact Berkovich adelic set $\bb E=(E_v)_{v\in M_{\bb Q}}$ is a collection of compact affinoids $E_v\subseteq\berkA$ for each place $v\in M_{\bb Q}$, with $E_v=\ovl{D}_{\an}(0,1)$ for all but finitely many $v$. The Robin constant $V(\bb E)$ of $\bb E$ is $V(\bb E):= \sum_{v\in M_{\bb Q}}V(E_v)$.
\end{definition}

\begin{theorem}\label{adelic-fekete-szego}
Let $\bb E=(E_v)_{v\in M_{\bb Q}}$ be a compact Berkovich adelic set. If $V(\bb E)\geq0$, then there exists infinitely many $x\in\ovl{\bb{Q}}$ such that the $\Gal(\ovl{\bb{Q}}/\bb Q)$-orbit of $x$ lies inside $E_v$ at every place $v$.
\end{theorem}

Our strategy is to construct, given any $\eps$-ordinary $(f,g)$, a compact Berkovich adelic set $\bb E$ with $V(\bb{E}) > 0$ and $\sum_{v} n_v \sup_{x \in E_v} \left( G_f(x) + G_g(x) \right)$ small. 

\begin{theorem}\label{upper-bound-most}
For any $\eps$-ordinary pair $(f,g)\in\PxP$, we have
\[\liminf_{x\in\Qbar}\left(\widehat h_f(x)+\widehat h_g(x)\right)\leq2\left(\mc C+O(d\eps+\tfrac1d)\right)\log X+O(1),\]
where $\mc C=\log\left(\frac{9+\sqrt{17}}8\right)+\frac{\sqrt{17}-1}8=0.88532\ldots$  
\end{theorem}

\begin{proof}
Let $\alpha:=\frac{\sqrt{17}-1}8$, and let $c>0$ be a small parameter to be chosen later. We consider the adelic set $\bb E=(E_v)_{v\in M_{\bb Q}}$ given by
\[E_v\coloneqq\begin{cases}
S_{f,v}^\cup&v\text{ assoc.\ }a_j\text{ with }0<\frac jd<\alpha+c,\\
S_{g,v}^\cup&v\text{ assoc.\ }b_j\text{ with }0<\frac jd<\alpha+c,\\
S_{f,v}^\cap&v\text{ assoc.\ }a_j\text{ with }2\alpha<\frac jd<1,\\
S_{g,v}^\cap&v\text{ assoc.\ }b_j\text{ with }2\alpha<\frac jd<1,\\
\ovl{D}_{\an}(0,1)&\text{otherwise},
\end{cases}\]    

with $S_{f,v}^\cup,S_{f,v}^\cap$ as defined in \cref{NonArchimedean4}. Using \cref{NonArchimedean4} and \cref{sum-assoc}, we bound $V(\bb E)$ from below:
$$
V(\bb E) =\sum_{v\in M_{\bb Q}}\!V(E_v)
=\sum_{\substack{v\text{ assoc.\ }a_j\\0<\frac jd<\alpha+c}}\frac1{2(d-j)}\log|a_j|_v+\!\!\sum_{\substack{v\text{ assoc.\ }b_j\\0<\frac jd<\alpha+c}}\frac1{2(d-j)}\log|b_j|_v$$ 
$$-\!\!\sum_{\substack{v\text{ assoc.\ }a_j\\2\alpha<\frac jd<1}}\frac1j\log|a_j|_v-\!\!\sum_{\substack{v\text{ assoc.\ }b_j\\2\alpha<\frac jd<1}}\frac1j\log|b_j|_v
\geq\left(2-O(d\eps)\right)\log X\!\!\sum_{0<\frac jd<\alpha+c}\frac1{2(d-j)}-2\log X\!\!\sum_{2\alpha<\frac jd<1}\frac1j.$$
We recognize the sums above as Riemann sums:
$$
\sum_{0<\frac jd<\alpha+c}\frac1{2(d-j)}\geq\int_0^{\alpha+c}\frac1{2(1-t)} d t-O\left(\frac1d\right)=\frac{\log(1-\alpha-c)}2-O\left(\frac1d\right),$$
$$
\sum_{2\alpha<\frac jd<1}\frac1j\geq\int_{2\alpha}^1\frac1t d t-O\left(\frac1d\right)=\log(2\alpha)-O\left(\frac1d\right).
$$
Hence
$$
V(\bb E) \geq\left(-\left(1-O(d\eps)\right)\log(1-\alpha-c)-2\log(2\alpha)-O(\tfrac1d)\right)\log X \geq\left(\Theta(c)-O(d\eps+\tfrac1d)\right)\log X,
$$
where we have used the identity $-\frac12\log(1-\alpha)=\log(2\alpha)$. Hence we may choose $c=O(d\eps+\frac1d)$ so that $V(\bb E)\geq0$. 

We want to apply the adelic Fekete--Szeg\H{o} theorem to $\bb E$, but the issue is that $S_{f,v}^\cup,S_{f,v}^\cap$ are not Berkovich affinoids. To fix this,  we construct a new adelic set $\widetilde{\bb E}=(\widetilde E_v)_{v\in M_{\bb Q}}$ by replacing the sets $S_{f,v}^\cup,S_{f,v}^\cap$ in the definition of $\bb E$ by compact affinoids $\widetilde S_{f,v}^\cup\supseteq S_{f,v}^\cup$ and $\widetilde S_{f,v}^\cap\supseteq S_{f,v}^\cap$ such that
\begin{equation} \label{eq:upper-bound-most-eta}
\sup_{\widetilde S_{f,v}^\cup}(G_{f,v}+\log^+|\cdot|_v) \leq\frac1{d-j}\log|a_j|_v+\eta, \qquad \sup_{\widetilde S_{f,v}^\cap}(G_{f,v}+\log^+|\cdot|_v) \leq\eta,
\end{equation}
where $\eta>0$ is a small parameter to be chosen later. (This is possible by \cref{NonArchimedean4}, and the fact that $G_{f,v}+\log^+|\cdot|_v$ is continuous on $\berkA$.) Similarly, we also replace $S_{g,v}^\cup,S_{g,v}^\cap$ with closed neighborhoods $\widetilde S_{g,v}^\cup,\widetilde S_{g,v}^\cap$. 

Now $\widetilde{\bb E}$ is a compact Berkovich adelic set and $V(\widetilde{\bb E})\geq V(\bb E)\geq0$, so by \cref{adelic-fekete-szego} we can find infinitely many distinct $x_n \in \Qbar$ whose $\Gal(\Qbar/\bb Q)$-orbit $F_n$ lies entirely in $\widetilde{\bb E}$.

Let $K_n$ be a number field containing $F_n$. The local contribution to $\widehat h_f(x_n)$ at $v\in M_{\bb Q}$ is
\[\frac1{[K_n:\bb Q]}\sum_{\substack{v'\in M_{K_n}\\v'\mid v}}\!\!n_{v'}G_{f,v'}(x_n)=\frac1{|F_n|}\sum_{x\in F_n}G_{f,v}(x)\leq\sup_{\widetilde E_v}G_{f,v},\]
and similarly for $g$. Hence we are left to bound $\sup_{\widetilde E_v}(G_{f,v}+G_{g,v})$ from above at good, bad, and archimedean places $v\in M_{\bb Q}$.

For good places $v\in M_{\bb Q}$, if $v$ is associated to $a_j$ then $g$ has good reduction at $v$, so $G_{f,v}+G_{g,v}=G_{f,v}+\log^+|\cdot|_v$, and we are in the setting of \cref{NonArchimedean4} and (\ref{eq:upper-bound-most-eta}). If $v$ is associated to $b_j$, we use the analogous statements with $f,g$ swapped. Working through the possible cases in terms of $j$, we have

\[\sup_{\widetilde E_v}(G_{f,v}+G_{g,v})\leq\begin{cases}
\frac1{d-j}\log|a_j|_v+\eta&v\text{ assoc.\ }a_j\text{ with }0<\frac jd<\alpha+c,\\
\frac1{d-j}\log|b_j|_v+\eta&v\text{ assoc.\ }b_j\text{ with }0<\frac jd<\alpha+c,\\
\frac1d\log|a_j|_v&v\text{ assoc.\ }a_j\text{ with }\alpha+c<\frac jd<2\alpha,\\
\frac1d\log|b_j|_v&v\text{ assoc.\ }b_j\text{ with }\alpha+c<\frac jd<2\alpha,\\
\eta&v\text{ assoc.\ }a_j\text{ or }b_j\text{ with }2\alpha<\frac jd<1.
\end{cases}\]

or bad places $v\in M_{\bb Q}$, we have by \cref{green-bound-nonarch} and \cref{pairing-bound-gauss},

$$
\sup_{\ovl{D}_{\an}(0,1)}(G_{f,v}+G_{g,v}) \leq G_{f,v}(\zeta(0,1))+G_{g,v}(\zeta(0,1)) =\frac1d(\log\mc M_{f,v}+\log\mc M_{g,v}).
$$
Finally, for the archimedean place, we have $\widetilde E_\infty=\del B(0,1)$, so for any $z\in\widetilde E_\infty$ and $w\in K_{f,\infty}$ we have $|z-w|\leq R_{f,\infty}+1$, with $R_{f,\infty}$ as in \cref{JuliaSetBound1Arch}. Thus
\[G_{f,\infty}(z)=\int\log|z-w|_{\infty} d\mu_{f,\infty}(w)\leq\log(R_{f,\infty}+1),\]
and similarly for $g$. Hence by \cref{rf-arch-bound-most} we have
\[\sup_{\del B(0,1)}(G_{f,\infty}+G_{g,\infty})\leq\log(R_{f,\infty}+1)+\log(R_{g,\infty}+1)\leq O(\eps\log X+1).\]

Now we choose $\eta\leq\#\{\text{good places }v\in M_{\bb Q}\}^{-1}$, so
$$
\widehat h_f(x_n)+\widehat h_g(x_n) \leq\sum_{\substack{v\text{ assoc.\ }a_j\\0<\frac jd<\alpha+c}}\frac1{d-j}\log|a_j|_v+\!\!\sum_{\substack{v\text{ assoc.\ }b_j\\0<\frac jd<\alpha+c}}\frac1{d-j}\log|b_j|_v $$
$$+\!\!\sum_{\substack{v\text{ assoc.\ }a_j\\\alpha+c<\frac jd<2\alpha}}\frac1d\log|a_j|_v+\!\!\sum_{\substack{v\text{ assoc.\ }b_j\\\alpha+c<\frac jd<2\alpha}}\frac1d\log|b_j|_v +\frac1d\sum_{v\text{ bad}}(\log\mc M_{f,v}+\log\mc M_{g,v})+O(\eps\log X+1).
$$
By \cref{bad-height-bound}, this is
$$ \leq\left(2\!\!\sum_{0<\frac jd<\alpha+c}\!\!\frac1{d-j}+2
\!\!\sum_{\alpha+c<\frac jd<2\alpha}\!\!\frac1d+O(d\eps)\right)\log X+O(1).$$
The first sum is a Riemann sum as before, and the second sum is easy to bound:
$$
\widehat h_f(x_n)+\widehat h_g(x_n) \leq\left(-2\log(1-\alpha-c)+2(\alpha-c)+O(d\eps+\tfrac1d)\right)\log X+O(1) $$
$$=2\left(\mc C+O\left(c+d\eps+\tfrac1d\right)\right)\log X+O(1),$$
where $\mc C=-\log(1-\alpha)+\alpha$, and we are done since $c=O(d\eps+\frac1d)$.
\end{proof}

Putting together \cref{lower-bound-most,upper-bound-most}, we obtain \cref{main-thm-bogomolov}. Recall that we define $\cal{P}$ and $\cal{P}_c$ to be the space of monic polynomials and monic centered polynomials with rational coefficients respectively. 

\mainthmbogomolov*

\begin{proof}
Let $\delta > 0$ be given. Then for all large enough $d$ depending on $\delta$, we can choose $\epsilon$ small enough depending on $d$ such that the terms $O(d \epsilon + \frac{1}{d}),O(d^2 \epsilon + \frac{\log d}{d})$ and $O(d \epsilon + \frac{1}{d})$ occurring in \cref{height-log-most}, \cref{lower-bound-most} and \cref{upper-bound-most} respectively are bounded above by $\frac{\delta}{3}$.
\par 
Let $X$ be sufficiently large such that for the terms $O(1)$'s occuring in \cref{lower-bound-most} and \cref{upper-bound-most}, we have $\frac{\delta}{3}\log X \geq O(1)$. Then by \cref{lower-bound-most} and \cref{upper-bound-most}, we know that for any $\eps$-ordinary pair $(f,g) \in \cal{S}(X)$ we have
$$2\left(\ln 2 - \frac{2}{3} \delta \right) \log X \leq \liminf_{x \in \ovl{\bb{Q}}} \left( \widehat{h}_f(x) + \widehat{h}_g(x) \right) \leq 2 \left(\cal{C} + \frac{2}{3} \delta \right) \log X.$$
Now \cref{height-log-most} implies that for an $\eps$-ordinary pair of $\cal{S}(X)$, we have
$$h(f) + h(g) \geq d \left(2 - \frac{1}{3} \delta \right) \log X.$$
We also have the trivial upper bound
$$h(f) + h(g) \leq 2d \log X$$
since each coefficient of $f$ and $g$ have height bounded by $\log X$ and we have at most $2d$ different coefficients. We now assume $\delta$ is small enough so that
$$\left(2 - \frac{1}{3} \delta \right) ( \cal{C} + \delta) \geq 2 \left(\cal{C} + \frac{2}{3} \delta \right).$$
This is possible since $\cal{C} < 1$. Then substituting in our bounds of $h(f) + h(g)$ against $\log X$, we obtain
$$(\ln 2 - \delta) (h(f) + h(g)) \leq d \liminf_{x \in \ovl{\bb{Q}}} \left( \widehat{h}_f(x) + \widehat{h}_g(x) \right) \leq (\cal{C} + \delta) (h(f) + h(g))$$
for all $\eps$-ordinary pairs $(f,g) \in \cal{S}(X)$, if $X$ is sufficiently large. By \cref{ordinary-most}, the proportion of $\eps$-ordinary pairs in $\cal{S}(X)$ is at least $1 - O_{\epsilon}(d^2 X^{-\epsilon})$. Clearly we can take $\epsilon = \frac{c}{d^2}$ for some constant $c > 0$ that depends on $\delta$ and so the proportion of $\eps$-ordinary pairs will be $1 - O_d(d^2 X^{-c/d^2})$. Finally, we set $\cal{T}$ to be the all pairs $(f,g)$ which are $\eps$-ordinary for some $X$ which is sufficiently large for our previous bounds to hold. Then 
$$\frac{|\cal{T} \cap \cal{S}(X)|}{|\cal{S}(X)|} \geq 1 - \frac{A}{X^{a}} $$
for some constants $a = a(d) > 0, A = A(d) > 0$ that depend on the degree $d$, and
$$(\ln 2 - \delta) (h(f) + h(g)) \leq d \liminf_{x \in \ovl{\bb{Q}}} \left( \widehat{h}_f(x) + \widehat{h}_g(x) \right) \leq (\cal{C} + \delta) (h(f) + h(g))$$
for all $(f,g) \in \cal{T}$ as desired.

\end{proof}

\printbibliography

\end{document}